\input ./style/arxiv-general.cfg
\documentclass[MSNbibl,number,citesort,dvips]{arxbj}
\makeatletter
   \@ifpackageloaded{graphicx}{}{\usepackage{graphicx}}
\makeatother
\usepackage{upgreek}

\aid{0}
\volume{21}
\issue{2}
\pubyear{2015}
\firstpage{1014}
\lastpage{1046}
\doi{10.3150/14-BEJ596} 

\makeatletter

\newcommand{\rrvert}{\vert}
\newcommand{\llvert}{\vert}

\newtheorem{lem}[thm]{Lemma}
\newtheorem{cor}[thm]{Corollary}
\newproclaim{defn}[thm]{Definition}

\newcommand{\Expect}{\mathbb{E}}
\newcommand{\eqref}[1]{(\ref{#1})}

\newcommand{\Prob}{\mathbb{P}}

\newcommand{\ball}{\operatorname{ball}}
\renewcommand{\d}{\mathrm{d}}

\newcommand{\sgn}{\operatorname{sgn}}

\renewcommand{\issue}[1]{}

\def\sfrac#1#2{#1/#2}

\def\afrac#1#2{#1/(#2)}

\def\sklfrac#1#2{(#1/#2)}
\makeatother

\begin{document}
\begin{frontmatter}

\title{Coupling, local times, immersions}
\runtitle{Coupling, local times, immersions}

\begin{aug}
\author[A]{\inits{W.S.}\fnms{Wilfrid S.} \snm{Kendall}\corref{}\ead[label=e1]{w.s.kendall@warwick.ac.uk.com}\ead[label=e2,url]{www.warwick.ac.uk/wsk}}
\address[A]{Department of Statistics, University of Warwick, Coventry
CV5 6FQ, UK.\\ \printead{e1,e2}}
\end{aug}

\received{\smonth{12} \syear{2012}}
\revised{\smonth{1} \syear{2014}}

%
\begin{abstract}
This paper answers a question of \'{E}mery [In \textit{S\'eminaire de Probabilit\'es {XLII}} (2009) 383--396 Springer] by constructing an
explicit coupling of
two copies of the Bene\v{s} \textit{et al.} [In \textit{Applied Stochastic Analysis} (1991) 121--156 Gordon \& Breach] diffusion (BKR
diffusion), neither of which starts at the origin, and whose natural
filtrations agree.
The paper commences
by surveying
probabilistic coupling,
introducing the formal definition of an \emph{immersed coupling} (the
natural filtration of each component is immersed in a common underlying
filtration; such couplings
have been described as \emph{co-adapted} or \emph{Markovian} in older
terminologies)
and of an \emph{{equi-filtration}} coupling (the natural filtration of
each component is immersed in the filtration of the other; consequently
the underlying filtration is simultaneously the natural filtration for
each of the two coupled processes).
This survey is
followed by a
detailed
case-study of the simpler but potentially thematic problem of coupling
Brownian motion together with its local time at $0$.
This problem possesses its own intrinsic interest as well as being
closely related to the BKR coupling construction.
Attention focusses on a simple
immersed (co-adapted) coupling, namely the reflection/synchronized coupling.
It is shown that this coupling is optimal amongst all immersed
couplings of
Brownian motion together with its local time at $0$, in the sense of
maximizing the coupling probability at all possible
times, at least when not started at pairs of initial points lying in a
certain singular set. However
numerical evidence indicates that the coupling is \emph{not} a maximal
coupling, and is a simple but non-trivial instance
for which this distinction occurs.
It is shown how the reflection/synchronized coupling
can be converted into a successful {{equi-filtration}} coupling, by
modifying the coupling
using a deterministic time-delay and then
by concatenating an infinite sequence of such modified couplings.
The construction of an explicit {equi-filtration} coupling of
two copies of the BKR diffusion follows by a direct generalization,
although the proof of success for the BKR coupling
requires somewhat more analysis than in the local time case.
\end{abstract}

%

\end{frontmatter}

\section{Introduction}\label{sec:introduction}
We begin with a brief survey of probabilistic coupling, which serves
both to introduce some key concepts and
to establish a context for the results proved in this paper.
The concept of coupling has a long and distinguished history, dating
back to Doeblin \cite{Doeblin-1938}
(a biographical appreciation is given by Lindvall
\cite{Lindvall-1991}).
The method is now the subject of two scholarly expositions (Lindvall \cite{Lindvall-1992},
Thorisson \cite{Thorisson-2000}),
and has become a standard tool
of the working probabilist
(a somewhat more general concept appears in ergodic theory as the
notion of a ``joining''). Historically the thematic problem for
coupling is that of constructing two copies of a given process
on the same sample space,
starting at two different starting points but eventually coinciding.
Such a coupling is said to be \emph{successful}.
In fact, many applications of coupling do not address the objective of
eventually coinciding; nevertheless the thematic
problem has been formative for the theory and remains significant in
developing methods and intuition.
Probabilistic coupling in general has found application throughout
probability, for example in construction of gradient estimates, in
distributional approximation (for instance, Stein--Chen approximation),
in perfect simulation, and in monotonicity results for heat equations
in insulated domains. The study of coupling in its own right is
therefore a foundational topic for probability theory.

A landmark development in the study of coupling was the introduction of
the notion of \emph{maximal coupling}: a coupling which simultaneously
maximises the chances of succeeding before time $t$ for all possible
$t$.
Perhaps it will surprise the reader to learn that maximal couplings
always exist: this was established by Griffeath
\cite{Griffeath-1975} for
time-homogeneous discrete Markov chains
and by Goldstein \cite{Goldstein-1978} for more
general discrete-time
processes, based on a tail $\sigma$-algebra condition.
(Note that even a maximal coupling need not necessarily have
probability $1$ of succeeding!)
See also the very explicit construction
given by Pitman \cite{Pitman-1976} for
time-homogeneous discrete Markov
chains, Sverchkov and Smirnov's \cite
{SverchkovSmirnov-1990} note on coupling for
continuous time
using the $J_1$ topology, and Thorisson's \cite
{Thorisson-1994} notion of
\emph{shift coupling}, which weakens the coupling requirement by
allowing for general time-shifting of the coupled processes. (An
informative treatment of some subtleties
is given in the treatment of ``faithful coupling'' in Rosenthal \cite{Rosenthal-1997}.)

In general, the construction of maximal couplings is a demanding
business: for
substantial applications the task of construction is liable to require
at least as much knowledge of
the process in question as might be needed to solve the original
problem to which the coupling method is to be applied.
(Notwithstanding this general and justifiable pessimism, the
simple \emph{reflection coupling} of Brownian motion is a successful
maximal coupling.
Attention was originally drawn to this construction by the influential
unpublished preprint of Lindvall \cite
{Lindvall-1982a}.) More commonly, one
works with less powerful couplings that are more easily constructed and
analyzed,
such as ``co-adapted couplings''. Co-adapted couplings (sometimes also
called ``Markovian couplings'' in the context
of coupling of Markov processes) require the two copies
of the processes concerned to be adapted to the same filtration, and to
have the same conditional laws based on conditioning on filtration $
\sigma$-algebras.
In the succinct language of filtrations (cf. Beghdadi-Sakrani and Emery
\cite{BeghdadiSakraniEmery-1999}, \'{E}mery \cite{Emery-2005,Emery-2009}),\issue{Vershik's
standardness criterion, see Vershik reference in
Emery-Schachermeyer in SemProb 35, nb erratum about lemma 17(?), and
relationship to co-adapted/{equi-filtration} coupling. See also
SemProb 38 article by Stephan Laurent (see also this author in Russian
journal) re eg I-coziness.}
the natural filtrations of the two processes must both be \emph
{immersed} in a common filtration (that is, the martingales of the
natural filtrations must remain martingales in the larger common filtration).
We therefore
propose and adopt the new terminology of \emph{immersed couplings} to
replace the nomenclature of
co-adapted or Markovian couplings:

\begin{defn}\label{def:immersed-coupling}
Consider two processes $X$ and $Y$. An \emph{immersed coupling} of
$X$ and $Y$ is a construction
of copies $\hat{X}$, $\hat{Y}$ of $X$, $Y$, defined on the
same probability space $(\Omega,\mathcal{F},\mathbb{P})$,
and adapted to the same filtration $\{\mathcal{F}_t\dvt t\geq0\}$, such
that any martingale in the natural
filtration of $\hat{X}$ remains a martingale in the common
filtration $\{\mathcal{F}_t\dvt t\geq0\}$, and
likewise for any martingale in the natural
filtration of $\hat{Y}$.
\end{defn}

The extent to which immersed couplings are less powerful than maximal couplings
was assessed in a preliminary way by
Burdzy and Kendall \cite{BurdzyKendall-2000}, where they
were studied in the guise of
Markovian couplings.
As part of a study of \emph{shy coupling} (the antithesis of the
thematic coupling problem, in which one seeks to construct coupled
copies which almost surely stay at least a fixed positive distance apart),
Kendall \cite{Kendall-2009a}, Lemma~6, records a
characterization of immersed
couplings of Brownian motion which has
long been part of the general folklore of stochastic calculus: any
immersed coupling of two $d$-dimensional Brownian motions $A$ and $
B$ can be represented by the stochastic differential equation
%
\begin{equation}
\label{eq:folklore} \d A = J^\top\,\d B + K^\top\,\d C,
\end{equation}
where $C$ is a Brownian motion independent of $B$ (perhaps
to be defined after augmenting the filtration, if this is necessary to
construct $C$), and $J$ and $K$ are two $(d\times d)$
matrix-valued predictable random processes satisfying $J^\top J+K^\top
K=\mathbb{I}$ where $\mathbb{I}$ is the $(d\times d)$
identity matrix.
We can view $J$ as a predictable matrix-valued control for a somewhat
degenerate stochastic control problem.
(An informal discussion of links between stochastic control and
coupling can be found in Kendall \cite{Kendall-2007}, Section~2.)

The terminology of immersed couplings is useful not only for its
succinct definition, but also because it draws attention to a stricter
constraint. Additionally,
one could demand that either of the coupled copies could be constructed
from the other, which corresponds to the requirement that the coupling possesses
the \emph{{equi-filtration}} property:

\begin{defn}\label{def:isofiltration-coupling}
Consider two processes $X$ and $Y$. An \emph{{equi-filtration}
coupling} of $X$ and $Y$ is an immersed coupling
$\hat{X}$, $\hat{Y}$ such that the natural filtration of $\hat
{X}$ is equal to that of~$\hat{Y}$.
\end{defn}

Of course it is the case that the {equi-filtration} coupling property
follows from
each natural filtration being immersed in the other.
Consider one of the simplest nontrivial examples of coupling;
Lindvall's Brownian reflection coupling is not only
a successful maximal coupling,
but also immersed and even {equi-filtration}. This is a very special
case; for example
Connor \cite{Connor-2007a}, Ph.D. thesis,
considers reflection coupling
of the Ornstein--Uhlenbeck process, if one copy of the
Ornstein--Uhlenbeck process is started from $0$
and the other copy is started from equilibrium.
He notes that reflection coupling of the driving Brownian motions is
clearly immersed but is not maximal even in the simple case.
(Further
exploration of the
difference between maximality and immersion for couplings can be found
in Kuwada and Sturm \cite{KuwadaSturm-2007}, Kuwada \cite
{Kuwada-2009}.)

The first objective of this paper
is to investigate
and explore properties of the construction of immersed and
{equi-filtration} couplings in the simple case of coupling Brownian
motion together with local time at $0$.
As a coupling problem this is only a little more
complicated than the basic Brownian motion case, but it produces an
example of existence of a
successful immersed coupling (the reflection/synchronized coupling,
Definition~\ref{coupling:reflection/synchronized})
which is optimal among
all immersed couplings but (according to numerical evidence) is not maximal
(Theorems \ref{thm:immersed-optimality}, \ref{thm:rates}, \ref
{thm:isofiltration-coupling} below). The reader may wish to compare
other work on optimal immersed couplings for random walks on the line,
on hypercubes and on hypercomplete graphs (Rogers \cite{Rogers-1999},
Connor and Jacka~\cite{ConnorJacka-2008}, Connor~\cite
{Connor-2009}).

A significant motivation for this study arises from the consideration
that the reflection coupling has been a model for a wide variety of
more sophisticated immersed couplings. For example, reflection coupling
has been generalized to the case of elliptic diffusions
with smooth coefficients (Lindvall and Rogers \cite{LindvallRogers-1986},
Chen and Li \cite{ChenLi-1989}),
and also to the case of Riemannian Brownian motion (Kendall \cite{Kendall-1998e}), in which case there are
connections with curvature
properties.
More recently, coupling techniques have been extended to cover some
cases of hypoelliptic diffusions
(Ben Arous \textit{et al}. \cite{BenArousCranstonKendall-1995},
Kendall and Price \cite{KendallPrice-2004}, Kendall \cite{Kendall-2007,Kendall-2009d});
essentially
the issue here is to couple simultaneously not only Brownian motion but
also one or more
path functionals of the Brownian motion, namely time integrals,
iterated time integrals,
and It\^o stochastic area integrals. Here it is necessary to augment
the reflection coupling strategy with
other coupling strategies, notably synchronous coupling and rotation coupling.
In the stochastic differential framework \eqref{eq:folklore},
synchronous coupling corresponds to $K=0$ and $J=\mathbb{I}$, while
rotation coupling corresponds to $K=0$ and $J$ equal to a $d$-dimensional rotation.
(It is interesting to compare this direction of research with the work
of \'{E}mery \cite{Emery-2005}, Theorem~1; this characterizes
Brownian filtrations using the notion of ``self-coupling'' -- jointly
immersed Brownian filtrations for which a prescribed scalar functional
is approximately coupled.)

While Brownian motion
together with local time at $0$ does not form a hypoelliptic
diffusion in the strict sense, nevertheless
the question of its coupling theory is clearly related to the
hypoelliptic couplings mentioned above.
The successful reflection/synchronized coupling is not only simple,
but also (in view of the results proved here)
evidently the right coupling for this situation.
It is reasonable to hope that a careful and complete study of the
reflection/synchronized coupling will be helpful
in formulating and studying coupling methods for more general
situations, as well as suggestive for coupling theory for hypoelliptic
diffusions.
The second objective of the paper is to demonstrate the first fruits of
this aspiration
and is fulfilled in Theorem~\ref{thm:BKR-isofiltration-coupling}
below, exhibiting a successful {equi-filtration} coupling for the BKR diffusion.
The approach follows closely the
methods developed for the reflection/synchronized coupling for
Brownian motion
together with local time at $0$.

In summary, then,
this paper conducts a case study of an
almost surely successful coupling of a simple non-elliptic diffusion in
the context of immersed and {equi-filtration} couplings; namely the
reflection/synchronized coupling for Brownian motion together with
local time at $0$. The results of this case study are then applied to
answer a question raised by \'{E}mery \cite
{Emery-2009}, by constructing an explicit
{equi-filtration} coupling for BKR diffusions neither of which are
begun at the origin.

Section~\ref{sec:brownian-local-time} introduces the simple reflection/synchronized coupling for Brownian motion together
with local time at $0$, exploiting Tanaka's formula and the L\'evy
transform to re-cast the problem in terms of coupling Brownian motion
together with a variant of its running supremum. The simplicity of this
coupling allows for explicit calculation: in particular
it is shown that the reflection/synchronized coupling is
optimal amongst all immersed couplings, at least when their starting
conditions are non-singular (here ``optimal'' means optimal in the
sense of
maximizing the probability of coupling by a given time $t$, for all
possible times $t$, while ``singular'' means that in the re-cast form
the two running suprema
processes do not start from the same level). The moment-generating function
for the coupling time is computed, and compared numerically with the
moment-generating function for the maximal coupling time:
numerical calculation
then indicates
that the reflection/synchronized coupling cannot be a maximal coupling.

The reflection/synchronized coupling is an immersed coupling but is
not {equi-filtration}. Section~\ref{sec:equi-filtration}
shows that if the couplings are perturbed by a simple deterministic
time delay then it is possible to use a sequence of the resulting
approximate couplings to construct a
successful {equi-filtration} coupling of Brownian motion together with
its local time at~$0$.

Section~\ref{sec:BKR} introduces the BKR diffusion, sketches the
immersed coupling described in \'{E}mery \cite{Emery-2009}
(which bears a strong family resemblance to the reflection/synchronized coupling of Section~\ref{sec:brownian-local-time},
and which therefore is described here as a variant reflection/synchronized coupling), and notes that significant components
of this variant reflection/synchronized coupling are actually
immersed in the natural filtrations of both coupled diffusions.
This is used to generate a successful {equi-filtration} coupling using
the strategy of Section~\ref{sec:equi-filtration}, hence answering
\'{E}mery's question.

The paper is concluded by Section~\ref{sec:conclusion}, which reviews
the results of the paper and discusses some further
research questions.

\section{Coupling Brownian motion together with local time}\label
{sec:brownian-local-time}
The purpose of this section is to exhibit a successful immersed (but
\emph{not} {equi-filtration}) coupling for the two-dimensional
diffusion made up of Brownian motion together with local time at zero.
The simple construction (known already to \'{E}mery
\cite{Emery-2009}) is
based on Tanaka's formula for Brownian
local time, and permits informative exact computations. In particular
we are able to prove optimality of this coupling amongst all immersed couplings
(Theorem~\ref{thm:immersed-optimality}), so long as the initial
conditions are non-singular in a manner to be explained below,
and thus to establish the optimal rate of immersed coupling (Theorem~\ref{thm:rates}).
We note in passing that this notion of optimality is distinct from the
notion of $\rho$-optimality introduced by Chen
\cite{Chen-1994}.

\subsection{Representation via the Tanaka formula}\label{sec:tanaka}
Recall the Tanaka formula or L\'evy transform, expressing Brownian
local time at $0$ in terms of a stochastic integral:
%
\begin{equation}
\label{eq:tanaka} \d|X| = \sgn(X)\,\d X + \d L^{(0)} .
\end{equation}
Here $X$ is a real Brownian motion and $L^{(0)}$ is the local time
accumulated by $X$ at $0$.
An immediate consequence of \eqref{eq:tanaka} is L\'evy's famous transform,
which represents $|X|$ and $L^{(0)}$ in terms of
a new real Brownian motion $B$ and $S$, a variant on the running
supremum of $B$:
%
\begin{eqnarray}\label{eq:supremum}
B &=& L^{(0)} - |X| ,
\nonumber
\\[-8pt]\\[-8pt]
S &=& L^{(0)} .\nonumber
\end{eqnarray}
It follows from \eqref{eq:supremum} that $B=L^{(0)}_0-|X_0|-\int\sgn
(X)\,\d X$ and $S_t=\max\{L^{(0)}_0, \sup\{B_s\dvt s\leq t\}\}$,
so $S$ does not start at $B_0$ if $|X_0|>0$.

Evidently it suffices to exhibit successful coupling strategies for $
(B, S)$; off the line $X=0$, this
L\'evy transform forms a $2\dvtx 1$ representation of
the original pair $(X, L^{(0)})$;
the two pre-images under the L\'evy transform meet together when the
Brownian motion $X$ hits $0$.

\subsection{The reflection/synchronized coupling
for immersed coupling of Brownian motion together
with local time}\label{sec:algorithm}
The above considerations show that it suffices to exhibit a successful
immersed coupling between
(a) the pair $(B,S)$ above
and (b) a copy $(\widetilde{B},\widetilde{S})$ started with
different initial conditions.
Were the corresponding $X$ and $\widetilde{X}$ not to agree at
coupling, one could simply continue with synchronized coupling
until $|X|=|\widetilde{X}|$ hits $0$. However, the reflection/synchronized coupling given below actually terminates with $B=S$ and
$\widetilde{B}=\widetilde{S}$,
so at the end of this coupling we already have $|X|=|\widetilde
{X}|=0$.
Without loss of generality, suppose that $B_0=L^{(0)}_0-|X_0|\geq
\widetilde{B}_0=\widetilde{L}^{(0)}_0-|\widetilde{X}_0|$.

\begin{defn}[(Reflection/synchronized coupling)]\label
{coupling:reflection/synchronized}
The \emph{reflection/synchronized coupling algorithm} consists of
two stages:
\begin{enumerate}[2.]
\item[1.]\emph{Reflection coupling} ($\d B=-\d\widetilde{B}$) till
the time $T_1=\inf\{t\dvt  B_t=\widetilde{B}_t\}$ (the first time that
$B$ and $\widetilde{B}$ meet); then
(if $(B, S)$ is not already coupled with $(\widetilde B, \widetilde
S)$).
\item[2.]\emph{Synchronized coupling} ($\d B=+\d\widetilde{B}$), run
from time $T_1$
until the time\vspace*{1pt}
$T_2=\inf\{t>T_1\dvt  B_t\equiv\widetilde{B}_t=S_{T_1}\vee\widetilde
{S}_0\}$
that $B\equiv\widetilde{B}$ first hits the higher level $
S_{T_1}\vee\widetilde{S}_0$ after time~$T_1$.
\end{enumerate}
\end{defn}

Note that at the end of stage $2$ we have $B=S$ and $\widetilde
{B}=\widetilde{S}$, so $|X|=|\widetilde{X}|=0$.

It is possible for the coupling of $(B,S)$ and $(\widetilde
{B},\widetilde{S})$ to be abbreviated to a one-stage (reflection)
coupling in case $S_0=\widetilde{S}_0$, for
if it happens that $B$ (and therefore $\widetilde{B}$)
both stay below $S_0=\widetilde{S}_0$ up to time $T_1$
then coupling will be successfully achieved at time $T_{\mathrm
{couple}}=T_1<T_2$. However we will see below that this case
can be viewed as singular, as a consequence of Lemma~\ref{lem:singular}.
Moreover if $X_0$ and $\widetilde{X}_0$ are of opposite sign then
they will not couple at this
stage: it still will be necessary to proceed to completion of the
synchronization stage so that
$B_{T_2}=S_{T_2}=\widetilde{B}_{T_2}=\widetilde{S}_{T_2}$
and therefore $X_{T_2}=\widetilde{X}_{T_2}=0$.

If not completed at the end of the reflection stage,
then the coupling will succeed at the time $T_{\mathrm{couple}}=T_2$; at that moment in time it is the case that simultaneously
$B=\widetilde{B}$ (since they are coupled by synchronization after
meeting at time $T_1$) and $S=\widetilde{S} (=B=\widetilde{B})$.
Note that the $[T_1, T_2]$ stage
depends on the behaviour of $B$ over the initial time interval $
[0,T_1]$.
Indeed,
note that by construction (and particularly by choice of initial
conditions\vspace*{2pt} $B_0=L^{(0)}_0-|X_0|\geq\widetilde{B}_0=\widetilde
{L}^{(0)}_0-|\widetilde{X}_0|$)
it is the case that $\widetilde{B}$ will stay below or equal to $
B$ until time $T_2$,
and hence $S_{T_1}\vee\widetilde{S}_0=S_{T_1}\vee\widetilde
{S}_{T_1}$.
The construction is illustrated in Figure~\ref{fig:coupling}.
%
\begin{figure}[b]

\includegraphics{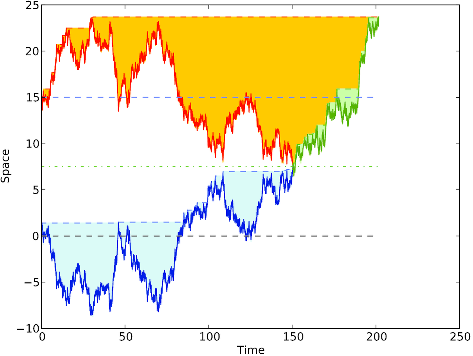}

\caption{Illustration of a successful reflection/synchronized
coupling of Brownian motion $B$ together with~$S$.}
\label{fig:coupling}
\end{figure}

The coupling is almost surely successful, since both the first and
second stages correspond to times taken for real Brownian motion to hit
specified levels. Indeed,
the coupling of $(B,S)$ and a copy $(\widetilde{B},\widetilde
{S})$ is {equi-filtration}, not just immersed,
because the stopping times $T_1$
and
$T_2$ can be rewritten as hitting times for $B$ in its natural filtration
(successively, from $B_0$ to $\tfrac12(B_0+\widetilde{B}_0)$,
then from $\tfrac12(B_0+\widetilde{B}_0)$
to $S_{T_1}\vee\widetilde{S}_0$),
and similarly also as hitting times for $\widetilde{B}$ in its
own natural filtration.
(In particular, $T_1$ can be rewritten as the hitting time of $
\widetilde{B}$ moving from $\widetilde{B}_0$
to $\tfrac{1}{2}(B_0+\widetilde{B}_0)$.)
The corresponding immersed coupling of $(X,L^{(0)})$ with $
(\widetilde{X},\widetilde{L}^{(0)})$
cannot be an {equi-filtration} coupling,
because the natural filtration of $X$ has to be augmented in order to
supply appropriate randomness for the signs of the excursions of $
\widetilde{X}$ from zero.
(See \'{E}mery \cite{Emery-2009}, Lemma~5, for a similar
augmentation in the
more complicated case of BKR diffusions.)

There is a natural reformulation of reflection/synchronized coupling
in terms of stochastic calculus:
set $\d\widetilde{B}=J\,\d B$ up to the coupling time $T_2$, where
the predictable control $J$ is given very simply by
%
\begin{equation}
\label{eq:co-adapted-coupling} J_t = \cases{ -1 & \quad \mbox{for} $t < T_1$
\mbox{(reflection stage)},
\cr
+1 &\quad  \mbox{for} $T_1 \leq t \leq
T_2$ \mbox{(synchronized stage)}. }
\end{equation}
The failure of mutual immersion for the coupling of $(X,L^{(0)})$
with $(\widetilde{X},\widetilde{L}^{(0)})$ is immediately apparent
from the
relevant stochastic differential equation
%
\begin{equation}
\label{eq:weak-not-strong} \d\widetilde{X} = \sgn(\widetilde{X}) J \sgn(X) \,\d X ,
\end{equation}
which is an instance of Tanaka's classic example of a Brownian motion $
\widetilde{X}$,
defined as a weak but not strong solution of a stochastic differential equation
driven by a second Brownian motion $\int J \sgn(X) \,\d X$.

\subsection{Optimality amongst immersed couplings}\label{sec:optimal}
The reflection/synchronized coupling strategy is faster than all
other immersed couplings, in the sense that it minimizes
\[
\Prob[T_{\mathrm{couple}} > t]
\]
simultaneously for all $t>0$, except perhaps for the singular case of
$S_0=\widetilde{S}_0$ (this
singular case is discussed around the statement of Lemma~\ref
{lem:singular} below).
Equivalently the distribution of the coupling time $T_{\mathrm
{couple}}$ for any immersed coupling exhibits stochastic domination
over the distribution
of $T_{\mathrm{couple}}$ for the reflection/synchronized coupling
(except perhaps in singular cases).

Before stating and proving a theorem which asserts this optimality, we
first establish some preparatory lemmas.
The first one concerns the coupling of two Brownian motions on $
[0,\infty)$
that are stopped when the first one of them hits $0$.

\begin{lem}\label{lem:quadrant-coupling}
Suppose the planar process $(U,V)$ is composed of two Brownian
motions which are related by an immersed coupling,
and
suppose that $(U,V)$ is started at a point $(U_0,V_0)$ in the
interior of the quadrant
$\{(u,v)\dvt u\geq0,v\geq0\}$. Let $T$ be the first time that $
(U,V)$ hits the boundary of the quadrant. Suppose it is desired to
construct the coupling so that $\Prob[(U_T,V_T)=(0,0)]=1$.
This is possible if and only if $U_0=V_0$ and the coupling is the
synchronized coupling.
\end{lem}

\begin{pf}
Using the formalism of It\^{o} \cite{Ito-1975} (see
also Ikeda and Watanabe \cite{IkedaWatanabe-1981}, Chapter
III.1, and the
development in Kendall \cite{Kendall-2001b}) and the
representation of immersed Brownian couplings given in \eqref{eq:folklore},
the general law of $(U,V)$ under an immersed coupling produces $\d
U^2=\d V^2=\d t$, $\operatorname{Drift}\d U=\operatorname{Drift}\d
V=0$,
and $\d U \,\d V=J\,\d t$ for an arbitrary adapted integrand $J\in
[-1,1]$ which can be viewed as the control
for the stochastic control problem of maximizing the objective function
$\Prob[(U_T,V_T)=(0,0)]$.

Without loss of generality, we may suppose that $U_0\geq V_0>0$.
Note that under reflection coupling ($J=-1)$ the probability of $
(U,V)$ hitting the diagonal $\{(u,v)\dvt u=v\}$ before time $T$ is
given by
\[
\Phi(U_0,V_0) = \frac{V_0}{\sklfrac{1}2(U_0+V_0)} .
\]
We extend the definition of $\Phi$ to the case $V_0\geq U_0>0$ by
setting $\Phi(U_0,V_0)=\Phi(V_0,U_0)$, so that
\[
\Phi(U,V) = \min \biggl\{\frac{U}{\sklfrac{1}2(U+V)}, \frac{V}{\sklfrac{1}2(U+V)} \biggr\} .
\]

An application of It\^o calculus shows that if $U>V>0$ then, under a
general control $J\in[-1,1]$,
\[
\operatorname{Drift}\d\Phi(U,V) = -\frac{2}{(U+V)^3}(U-V) (1+J)\,\d t ,
\]
and this is non-positive, and vanishes only when $J=-1$.
A similar result holds for $V>\linebreak[4] U>0$. On the other hand, if $
U_0=V_0>0$ and $J=+1$
then $(U,V)$ stays on the diagonal, so that $\Phi(U,V)$ then
remains constant.
An argument using the It\^o--Tanaka formula for semimartingales thus
shows that
$\Phi(U,V)$ is a supermartingale for all immersed couplings of $U$
and $V$, and becomes a martingale only under the strategy
``use reflection coupling till $(U,V)$ hits the diagonal or the
boundary, then use synchronized coupling till $(U,V)$ hits the boundary''.
It follows that $\Phi(U_0,V_0)$ is the maximum of $\Prob
[(U_T,V_T)=(0,0)]$ over all immersed couplings, and is attained only
by using this strategy.
The lemma follows.
\end{pf}

As a consequence of the lemma, we can prove the optimality of the
reflection/synchroni\-zed coupling
in the special case when $B_0=\widetilde{B}_0$.
This allows us to restrict attention to immersed couplings which
preserve the ordering of $B$ and $\widetilde{B}$.

\begin{lem}\label{lem:monotonicity}
Consider the reflection/synchronized coupling (Example~\ref
{coupling:reflection/synchronized})
for the special case $B_0=\widetilde{B}_0$.
This is the only optimal coupling amongst all immersed couplings
started with $B_0=\widetilde{B}_0$, and so
(since the reflection stage succeeds immediately)
the uniquely optimal way to proceed
is to cease immediately if $S_0=\widetilde{S}_0$, and
otherwise to conduct a synchronized coupling of $B$ and $\widetilde
{B}$ until $B\equiv\widetilde{B}$ hits $S_0\vee\widetilde{S}_0$.
\end{lem}

\begin{pf}
If $S_0=\widetilde{S}_0$, then coupling succeeds immediately and
there is nothing to prove.
Suppose without loss of generality that $S_0>\widetilde{S}_0$.
Coupling cannot succeed earlier than the first time
$\widetilde{T}$ at which $\widetilde{B}$ hits $S_0$,
and if we employ synchronized coupling then coupling will succeed at
this hitting time.

This shows that synchronized coupling is optimal, but we require strict
optimality. Consider a second coupling which does not employ
synchronized coupling throughout.
It then follows that there must be a moment, \emph{before} time $
\widetilde{T}$,
at which either $B>\widetilde{B}$ or $\widetilde{B}>B$.
We can apply Lemma~\ref{lem:quadrant-coupling} to $U=S_0-B$ and $
V=S_0-\widetilde{B}$; it follows that if synchronized coupling is not
employed right up to time $\widetilde{T}$,
then there is a positive probability that one of two possible cases has
occurred:
either $B$ has already hit $S_0$ by time $\widetilde{T}$, or $
B$ has not yet hit $S_0$ by time $\widetilde{T}$.
In the first case, the properties of Brownian motion $B$ show that
almost surely $S_{\widetilde{T}}>S_0$, and so successful coupling
must occur after $\widetilde{B}$ travels from $S_0$ to $
S_{\widetilde{T}}$.
In the second case, successful coupling must wait at least until $
\widetilde{B}$ and $B$ meet after time $\widetilde{T}$.

It follows that, for any coupling other than synchronized coupling, (a)
the coupling time can be no less than $\widetilde{T}$, (b) there is
a positive chance of it being
strictly greater than $\widetilde{T}$. This establishes the required
stochastic domination (since $\widetilde{T}$ is the hitting
time of a real Brownian motion started at $B_0=\widetilde{B}_0$ and
rising to $S_0>S_0\vee\widetilde{S}_0$) and so
the lemma follows.
\end{pf}

In passing, we are now able to explain the reason why it is appropriate
to describe as\vspace*{1pt} singular
the case when $B_0, \widetilde{B}_0 < S_0=\widetilde{S}_0$.
In this case it is possible for full coupling of $(B,S)$ with $
(\widetilde{B},\widetilde{S})$ to succeed as soon as $B$ first
meets $\widetilde{B}$, so long as
$B$ and $\widetilde{B}$ do not hit $S_0=\widetilde{S}_0$ (as
noted in Section~\ref{sec:algorithm}, this need not imply
success of the coupling of $X$ with $\widetilde{X}$ at that time).
The next lemma shows that
if $S_0\neq\widetilde{S}_0$ then this early success cannot occur.

%
\begin{lem}\label{lem:singular}
Suppose $S_0\neq\widetilde{S}_0$.
Then an optimal immersed coupling of $(B,S)$ and $(\widetilde
{B},\widetilde{S})$
succeeds exactly at the first time when $B$, $S$, $\widetilde{B}$, and $\widetilde{S}$
simultaneously coincide.
\end{lem}

%

\begin{pf}
Consider first the case $B_0=\widetilde{B}_0$. As shown by Lemma~\ref{lem:monotonicity},
it is then the case that the only optimal immersed coupling is provided
by synchronized coupling until $B=\widetilde{B}$
first hits $S_0\vee\widetilde{S}_0$, and the characterization of
coupling by simultaneous coincidence is immediate.

Consider the case $\widetilde{B}_0<B_0$ (the case of $\widetilde
{B}_0<B_0$ is entirely similar).
It is a consequence of Lemma~\ref{lem:monotonicity} that optimality of
the immersed coupling
implies that the relationship $\widetilde{B}\leq B$ must persist
till full coupling is successful.

So further suppose that $\widetilde{S}_0<S_0$. In that sub-case it
follows from $\widetilde{B}\leq B$ that
the relationship $\widetilde{S}\leq S$ must persist till full
coupling is successful. Coupling cannot succeed till
$\widetilde{S}$ hits $S$, and when that happens we must have $
\widetilde{B}=\widetilde{S}$. But in this sub-case we also have\vspace*{1pt}
$\widetilde{B}\leq\widetilde{S}\leq S$ and $\widetilde{B}\leq
B\leq S$. Consequently full coupling must succeed when
$\widetilde{S}$ first hits $S$, and at that time $B$, $S$, $
\widetilde{B}$, and $\widetilde{S}$
simultaneously coincide.

Suppose on the other hand that $S_0<\widetilde{S}_0$. In that
sub-case again, full coupling cannot succeed
before $S$ hits $\widetilde{S}$, at which time it is necessary
that $B$ also hits $\widetilde{S}$.
If it is further the case that $\widetilde{B}=B$ at that time, then
full coupling succeeds in the manner prescribed by the lemma.
If on the other hand $\widetilde{B}<B$ at that time, then (by the
properties of Brownian motion) there are instants immediately after this
time at which $\widetilde{B}<\widetilde{S}<B < S$, and we can
proceed as above.
\end{pf}


We shall now show that the distribution of the coupling
time $T_{\mathrm{couple}}$ under any immersed coupling can be
dominated in the limit (as $N\to\infty$)
by the distribution of $T_{\mathrm{couple}}$ under an immersed
coupling whose predictable control $J$ is restricted to values $\pm
1$
(thus, a ``bang-bang'' control),
and moreover such that $J$ is constant
on stochastic intervals $[\tau^{(N)}_k,\tau^{(N)}_{k+1})$ defined
as follows.
For any positive even integer $N>0$, consider the one-dimensional lattice
$\mathcal{L}^{(N)}$ (depending implicitly on $B_0$, and $
\widetilde{B}_0$)
\[
\mathcal{L}^{(N)} = B_0+\frac{\widetilde{B}_0-B_0}{N}\mathbb{Z} =
\biggl\{B_0 + \frac{k}{N}(\widetilde{B}_0-B_0)
\dvt  k = 0, \pm1, \pm2, \ldots\biggr\} .
\]
We define a \emph{mesh}, a sequence of stopping times $0=\tau
^{(N)}_0<\tau^{(N)}_1<\tau^{(N)}_2<\cdots\,$, as
a sequence of ``crossing times'' for this lattice:
\[
\tau^{(N)}_{k+1} = \inf \bigl\{t>\tau^{(N)}_k
\dvt  B_t \in\mathcal {L}^{(N)}\setminus\{B_{\tau^{(N)}_k}\}
\bigr\} .
\]
Sampling using this mesh of stopping times has the effect of
discretizing the Brownian motion $B$ into a random walk with steps $
\pm\tfrac{1}{N}(\widetilde{B}_0-B_0)$.

Note, for an immersed coupling restricted to a control $J$ which is
locally constant with $J=\pm1$ on each stochastic interval $[\tau
^{(N)}_k,\tau^{(N)}_{k+1})$ of the mesh:
\begin{enumerate}[3.]
\item[1.] both $B$ and $\widetilde{B}$, when sampled at times $0=\tau
^{(N)}_0<\tau^{(N)}_1<\tau^{(N)}_2<\cdots\,$, belong to the lattice $
\mathcal{L}^{(N)}$,
since $\widetilde{B}$ is obtained from $B$ using a predictable
control $J$ formed from synchronizations and reflections and which
alters only when $B$ belongs to the lattice;
\item[2.] because $N$ is even, $T_{\mathrm{couple}}$ belongs to the
set $\{\tau^{(N)}_0, \tau^{(N)}_1, \tau^{(N)}_2, \ldots\}$;
\item[3.] finally, our candidate for optimality, the reflection/synchronized coupling (Example~\ref{coupling:reflection/synchronized}),
can itself be viewed as one of these couplings, since the control
changes from $+1$ to $-1$ exactly at one of the stopping times in
the mesh.
(This is the reason why it is convenient to work with meshes of
stopping times, rather than decompositions of the time axis into
disjoint dyadic intervals.)
\end{enumerate}

We can now summarize and prove a result stating that an optimal
immersed coupling can be approximated in distribution by appropriately
chosen ``bang-bang'' controls of the above form.
The proof is related to the method of proof of
\'{E}mery \cite{Emery-2005}, Proposition~2; however
here we need the control $
J$ to have the ``bang-bang'' property rather
than simply to be locally constant, and to be composed of stopping
times drawn from a mesh of stopping times as specified above.

%
\begin{lem}\label{lem:approx}
For any fixed $t>0$, any optimal immersed coupling of $(B,S)$ and $
(\widetilde{B},\widetilde{S})$ can be approximated weakly over $
[0,t]$ (when viewed as a
probability distribution on the metric space of $4$-dimensional
continuous trajectories, equipped with the sup-norm) by
``bang-bang'' immersed couplings for which the
control $J$ takes values $\pm1$ only, and only changes at hitting
times belonging to some mesh.
\end{lem}

Note that the lemma does \emph{not} assert that the ``bang-bang''
couplings are successful!

\begin{pf*}{Proof of Lemma~\ref{lem:approx}}
Consider a general immersed coupling determined by $\widetilde
{B}=\widetilde{B}_0+\int J \,\d B$ and
subject to the constraint that the coupling is synchronized once $B$
and $\widetilde{B}$ have met.
(By Lemma~\ref{lem:monotonicity}, all optimal immersed couplings must
be of this form.)
Since $|J|\leq1$, for each $t>0$ we have $\Expect[\int_0^t J^2
\,\d s]<\infty$, and moreover
for each $\varepsilon>0$ we may find
continuous predictable $f$ with $\Expect[\int_0^t |f-J|^2 \,\d
s]<\varepsilon^2/4$ (for example, $f_t=\tfrac{2}{\delta^2}\int_{t-\delta}^t (t-s) J_s\,\d s$
for sufficiently small $\delta$). It then follows that for
sufficiently large $N$ we may approximate $f$ in $L^2$
by \emph{piece-wise constant} $J^{[c]}$ such that $J^{[c]}\in
[-1,1]$ is predictably constant on each dyadic interval
$[\tau^{(N)}_k,\tau^{(N)}_{k+1})$ of the mesh,
and $\Expect[\int_0^t |f-J^{[c]}|^2 \,\d s]<\varepsilon^2/4$, hence
\[
\Expect\biggl[\int_0^t \bigl|J-J^{[c]}\bigr|^2
\,\d s\biggr]<\varepsilon^2 .
\]
Doob's submartingale inequality then implies that we can control
\[
\sup_{s\leq t} \biggl\{\biggl\llvert \int_0^s
J\,\d B - \int_0^s J^{[c]}\,\d B\biggr
\rrvert \biggr\} ,
\]
so that $\widetilde{B}_0+\int J^{[c]}\,\d B$ is a good path-wise
approximation to
$\widetilde{B}=\widetilde{B}_0+\int J\,\d B$.\vspace*{1pt}

While $J^{[c]}$ is piecewise-constant on stochastic intervals related
to the mesh, it does not take values in $\{\pm1\}$.
We need an approximation based on a ``bang-bang'' control $J^{[bb]}$,
which is constrained by $J^{[bb]}\in\{\pm1\}$ as well
as by the requirement that $J^{[bb]}$ is predictably constant on
stochastic intervals
$[\tau^{(M)}_k,\tau^{(M)}_{k+1})$ which now must be formed on a new
mesh, defined for some still larger even integer $M=2^r N$, for
an integer $r>0$. Given $M>N$, we define
$J^{[bb];(M)}$ to ``track'' $J^{[c]}$ in the following co-adapted way:
%
\begin{equation}
\label{eqn:tracking} J^{[bb];(M)}_{\tau^{(M)}_k} = \cases{ +1 & \quad \mbox{if} $\displaystyle \int
_0^{\tau^{(M)}_k} J^{[bb];(M)}\, \d u \leq\displaystyle \int
_0^{\tau^{(M)}_k} J^{[c]} \,\d u $,
\cr
-1 &\quad
\mbox{if} $\displaystyle \int_0^{\tau^{(M)}_k} J^{[bb];(M)} \,\d u >
\displaystyle \int_0^{\tau^{(M)}_k} J^{[c]} \,\d u$. }
\end{equation}
Since $|J|\leq1$, it follows that we have the following bound for $
s\in[0,t]$:
%
\begin{equation}
\label{eq:integrand-bound} \biggl\llvert \int_0^s
J^{[c]} \,\d u - \int_0^s
J^{[bb];(M)} \,\d u\biggr\rrvert \leq 2 \sup \bigl\{ \bigl(
\tau^{(M)}_{k+1}\wedge t\bigr) - \bigl(\tau^{(M)}_{k}
\wedge t\bigr) \dvt  k=1, 2, \ldots \bigr\} ,
\end{equation}
converging almost surely to zero as $2^r=M/N\to\infty$.

Consider the sequence of two-dimensional processes $\{(B, \widetilde
{B}_0+\int J^{[bb];(M)}\,\d B) \dvt  M=2^r N\}$,
defined on the time-range $[0,t]$.
The one-dimensional coordinate processes being Brownian motions, it
follows that this sequence is tight.
Any convergent subsequence converges to a limit for which the
one-dimensional coordinate processes are Brownian motions;
moreover, using \eqref{eq:integrand-bound},
we may deduce that in the limit the product of the pair of
one-dimensional coordinate processes is equal to the sum of a martingale
and the integral $\int_0^s J^{[c]}\, \d u$. Hence by semimartingale
It\^o calculus the limit has the law of
$(B, \widetilde{B}_0+\int J^{[c]}\,\d B)$, no matter what convergent
subsequence is chosen, and therefore by
the theory of weak convergence we may deduce that the sequence of
random paths
$(B, \widetilde{B}_0+\int J^{[bb];(M)}\,\d B)$ converges weakly to
this limit.

It follows that we can choose a sequence of ``bang-bang'' controls $
J^{(n)}$, constant on
appropriate meshes
$\{[\tau^{(M_n)}_k,\tau^{(M_n)}_{k+1}) \dvt  k=1,2,\ldots\}$ (with $
M_n\to\infty$),
such that $(B, \widetilde{B}_0+\int J^{(n)}\,\d B)$
converges weakly (using supremum norm over the time interval $[0,t]$)
to the immersed coupling $(B, \widetilde{B})$ which was originally
under consideration.
\end{pf*}

We can now argue for optimality of the reflection/synchronized
coupling in the general\vspace*{1pt} non-singular case ($S_0\neq\widetilde{S}_0$).
We need only consider the case
when $B_0\neq\widetilde{B}_0$, since Lemma~\ref{lem:monotonicity}
covers the case of $B_0=\widetilde{B}_0$;
indeed it suffices to
consider only those immersed couplings which are constrained to be
synchronized couplings once $B$ and $\widetilde{B}$ have met.
Employing the terminology of Section~\ref{sec:algorithm}, we set $
T_1=\inf\{t:B_t=\widetilde{B}_t\}$. Thus, we need consider only
those immersed couplings for which $J_t=1$
once $t>T_1$.

%
\begin{thm}\label{thm:immersed-optimality}
Suppose that $B_0\neq\widetilde{B}_0$ and $S_0\neq\widetilde
{S}_0$.
The reflection/synchronized coupling (Example~\ref
{coupling:reflection/synchronized}) is optimal amongst
all immersed couplings of Brownian motion together with local time.
\end{thm}

\begin{pf}
As noted above, by Lemma~\ref{lem:monotonicity} we may restrict
attention to immersed couplings for which (without loss of generality)
$B\geq\widetilde{B}$, and such that $B\equiv\widetilde{B}$
after $T_1=\inf\{s\dvt  B_s=\widetilde{B}_s\}$.
Moreover, by the argument of Lemma~\ref{lem:singular}, at the coupling
time $T_{\mathrm{couple}}$ we must
have $B_{T_{\mathrm{couple}}}=S_{T_{\mathrm{couple}}}=\widetilde
{B}_{T_{\mathrm{couple}}}=\widetilde{S}_{T_{\mathrm{couple}}}$.

The first step is to use the weak approximations $(B, \widetilde
{B}_0+\int J^{(n)}\,\d B)$ of $(B, \widetilde{B})$
(as given in Lemma~\ref{lem:approx})
to build \emph{successful} immersed couplings of $(B,S)$ and $
(\widetilde{B},\widetilde{S})$ with coupling times
which are in the limit stochastically dominated by the coupling time
derived from $(B, \widetilde{B})$.
For convenience, we employ the Skorokhod representation of weak convergence;
augmenting the probability space if necessary,
we construct a copy\vspace*{1pt}
$(B^{(n)}, \widetilde{B}^{(n)}=\widetilde{B}_0+\int J^{*,(n)}\,\d
B^{(n)})$ of $(B, \widetilde{B}_0+\int J^{(n)}\,\d B)$ on the
same probability space as $(B, \widetilde{B})$ such that almost surely\vspace*{1pt}
$B^{(n)}\to B$ and $\widetilde{B}^{(n)}\to\widetilde{B}$
uniformly on the time interval $[0,t]$.
(We note in passing that this construction need not respect the
underlying filtration. The stochastic integrand
$J^{*,(n)}$ and the stochastic integral $\int J^{*,(n)}\,\d B^{(n)}$
are defined with respect to the natural
filtration of $B^{(n)}$, which need not immerse in the original filtration!)

Although the target coupling of $(B,S)$ and $(\widetilde
{B},\widetilde{S})$ couples at $T_{\mathrm{couple}}$, we should not
suppose
that $(B^{(n)},S^{(n)})$ and $(\widetilde{B}^{(n)}, \widetilde
{S}^{(n)})$ couple at this time
(using $S^{(n)}$ and $\widetilde{S}^{(n)}$ to denote the
corresponding supremum processes). However,
we can modify $(B^{(n)}, \widetilde{B}^{(n)})$ to produce a coupling
which does not succeed much later than the original coupling.

Indeed, we have restricted attention to immersed couplings such that at
the coupling time $T_{\mathrm{couple}}$ we have
$B_{T_{\mathrm{couple}}}=S_{T_{\mathrm{couple}}}=\widetilde
{B}_{T_{\mathrm{couple}}}=\widetilde{S}_{T_{\mathrm{couple}}}$.
Accordingly, we may choose a sequence $\varepsilon_n\to0$ such that
\[
\Prob\bigl[B^{(n)}_{T_{\mathrm{couple}}}, S^{(n)}_{T_{\mathrm{couple}}},
\widetilde{B}^{(n)}_{T_{\mathrm{couple}}}, \widetilde {S}^{(n)}_{T_{\mathrm{couple}}}
\mbox{ all lie within }\pm\varepsilon_n\mbox{ of each other}\bigr]
\geq1-\varepsilon_n .
\]
Accordingly, if we set
\[
T^{(n)}_3 = \inf \bigl\{s \dvt  B^{(n)}_{s},
S^{(n)}_{s}, \widetilde{B}^{(n)}_{s},
\widetilde{S}^{(n)}_{s} \mbox{ all lie within }\pm
\varepsilon_n\mbox{ of each other} \bigr\} ,
\]
then
\[
\Prob\bigl[T^{(n)}_3>t\bigr] \leq\Prob[T_{\mathrm{couple}}>t]+
\varepsilon_n .
\]

But at time $T^{(n)}_3$ we can modify the construction of $
(B^{(n)},S^{(n)})$ and $(\widetilde{B}^{(n)}, \widetilde{S}^{(n)})$
to use the reflection/synchronized coupling
(Example~\ref{coupling:reflection/synchronized}), obtaining
successful coupling at time $T^{(n)}_{\mathrm{couple}}\geq T^{(n)}_3$.
As $\varepsilon\to0$ so we can deduce by scaling that the extra
time $T^{(n)}_{\mathrm{couple}}-T^{(n)}_3$
required for success of this final coupling must tend to zero in probability.

It follows from these arguments that the infimum of the probability of
failing to couple before time $t$, for any fixed $t>0$,
\[
\Prob[T_{\mathrm{couple}} > t] ,
\]
can be approached by considering $\Prob[T_{\mathrm
{couple}}>t+\varepsilon_n]$ for suitable $\varepsilon_n\to0$ and\vspace*{2pt}
immersed couplings based on ``bang-bang'' controls $J^{[bb]}$
constrained by change only at stopping times
taken from meshes
$0=\tau^{(N)}_0<\tau^{(N)}_1<\tau^{(N)}_2<\cdots\,$, and which
become synchronous after $B$ and
$\widetilde{B}$ first meet.

Consider such an immersed coupling with control $J^{[bb]}$. For a
fixed $t>0$, we consider the following value function, defined for $
0\leq u<t$:
%
\begin{eqnarray}
\label{eqn:value-function} &&\hspace*{-15pt}V(u;b,\widetilde{b},s,\widetilde{s})\nonumber\\[-8pt]\\[-8pt]
&&\hspace*{-15pt}\quad  =
\Prob[T_{\mathrm{couple}} > t-u | B_u=b, \widetilde{B}_u=
\widetilde {b},S_u=s,\widetilde{S}_u=\widetilde{s};
\mbox{reflection/synchronized coupling}] .\nonumber
\end{eqnarray}
We are particularly interested in the discrete-time process obtained by
sampling at stopping times taken from the specified mesh, but stopping
at the terminal time $t$:
\[
\bigl\{Z_n = V\bigl(\tau^{(N)}_n \wedge t;
B_{\tau^{(N)}_n \wedge t}, \widetilde{B}_{\tau^{(N)}_n \wedge t}, S_{\tau^{(N)}_n \wedge t},
\widetilde{S}_{\tau^{(N)}_n \wedge t}\bigr) \dvt  n = 0, 1, 2, \ldots \bigr\} .
\]
It follows by definition that $V(u; B_u, \widetilde{B}_u, S_u,
\widetilde{S}_u)$ is a bounded martingale under the reflection/synchronized coupling.
Under this coupling
$Z$ is a discrete-time martingale since it is obtained from the
bounded process $\{V_{u\wedge t}\dvt u\geq0\}$ by sampling at stopping times.
We shall
show that $Z$ is a supermartingale under the coupling specified by $
J^{[bb]}$, and moreover
that the martingale property cannot hold if $J^{[bb]}=+1$ over the
initial time interval $[0,\tau^{(N)}_1)$.
Arguing inductively, this suffices to establish the theorem.

The crux of the matter is to consider the behaviour of the value
function at time zero if the initial segment of coupling is
synchronized. To this end,
we make a special construction of the reflection/synchronized
coupling referred to by the value function: we suppose two independent
Brownian motions
are employed (both begun at $0$), namely $B^{(r)}$ to drive the
reflection stage of the coupling, and $B^{(s)}$ to drive the
synchronized stage.
%
We set
\begin{itemize}
\item
$\tau^{(N;s)}_1=\inf\{t>0\dvt |B^{(s)}_t|=\frac{1}{N}(B_0-\widetilde
{B}_0)\}$,
\item
$T^{(r)}_1$ to be the time when $B^{(r)}$ first hits $-\tfrac
{1}{2}(B_0-\widetilde{B}_0)$, corresponding to the end of the
reflection stage,
\item
and $M^{(r)}=\sup\{B^{(r)}_s\dvt s\leq T^{(r)}_1\}+B_0$ to be the
maximum level achieved during the reflection stage.
\end{itemize}
Then the reflection/synchronized coupling time corresponds in law
to $T^{(*)}_3=T^{(r)}_1+\inf\{s\dvt B^{(s)}_s +\frac
{1}{2}(B_0+\widetilde{B}_0)= \max\{S_0\vee\widetilde{S}_0,M^{(r)}\}\}
$,
and so $Z_0=\Prob[T^{(*)}_3>t]$.

Now consider the effect of commencing with a session of synchronized coupling.
We consider two possible cases.
Suppose in the first case that
\[
M^{(r)} - \frac{1}{N}(B_0-\widetilde{B}_0)
\leq S_0\vee\widetilde {S}_0 .
\]
Then we can represent the initial session of synchronized coupling by
using $B^{(s)}|_{[0,\tau^{(N;s)}_1)}$,
and then replacing $B^{(s)}_t$ by $B^{(s)}_{t+\tau
^{(N;s)}_1}-B^{(s)}_{\tau^{(N;s)}_1}$. Evidently
the distribution of $T^{(*)}_3$ is unaffected by this change. If furthermore
\[
M^{(r)} + \frac{1}{N}(B_0-\widetilde{B}_0)
\leq S_0\vee\widetilde{S}_0
\]
then the reflection stage will start at time $\tau^{(N;s)}_1$ at
level $B^{(s)}_{\tau^{(N;s)}_1}$,
moreover by the end of the reflection stage the
supremum of the coupled processes will not exceed $S_0\vee\widetilde
{S}_0$, and the subsequent
synchronization stage will have to move from
$B^{(s)}_{\tau^{(N;s)}_1}+\tfrac{1}{2}(B_0+\widetilde{B}_0)$
to $S_0\vee\widetilde{S}_0$.
It follows that $T^{(*)}_3$ is still the coupling time.
If on the other hand
\[
M^{(r)} + \frac{1}{N}(B_0-\widetilde{B}_0)
> S_0\vee\widetilde{S}_0 ,
\]
then there is a possibility that the supremum of the coupled processes
\emph{will} exceed $S_0\vee\widetilde{S}_0$. However, the subsequent
synchronization stage will still have to move from
$B^{(s)}_{\tau^{(N;s)}_1}+\tfrac{1}{2}(B_0+\widetilde{B}_0)$
to $S_0\vee\widetilde{S}_0$, but may have to move even further.
Thus, the coupling time still cannot occur earlier than $T^{(*)}_3$.

Suppose in the second case that
\[
M^{(r)} - \frac{1}{N}(B_0-\widetilde{B}_0)
> S_0\vee\widetilde{S}_0 .
\]
Then we can represent the initial session of synchronized coupling by
using an independent copy $\hat B|_{[0,\hat\tau^{(N;s)}_1)}$ of $
B|_{[0,\tau^{(N;s)}_1)}$, and restarting the construction at the new
starting points\vspace*{1pt}
$B_0\pm\tfrac{1}{N}(B_0-\widetilde{B}_0)$, $\widetilde{B}_0\pm
\tfrac{1}{N}(B_0-\widetilde{B}_0)$ (same sign for each initial
increment). Regardless of the sign of the initial increment, coupling
occurs at
$\hat{\tau}^{(N;s)}_1 + T^{(*)}_3$, so is delayed relative to $
T^{(*)}_3$.

It follows from these arguments that an initial session of synchronized
coupling followed by reflection/synchronized coupling cannot increase
the probability of
successful coupling by time $t$
compared with that of reflection/synchronized coupling,
and has a positive chance of reducing it;
consequently $Z$ is a supermartingale and the martingale property for
$Z$ cannot hold if $J^{[bb]}=+1$ over the initial time interval $
[0,\tau^{(N)}_1)$.

Since $Z$ is a martingale for the reflection/synchronized coupling,
this suffices to establish the theorem.
\end{pf}

It is likely that the reflection/synchronized coupling is the \emph
{unique} optimal immersed coupling, but we do not pursue this
technicality here.

In the singular case the value function \eqref{eqn:value-function}
takes on a more complicated form, and
the above approach in this case will no longer settle
whether or not the reflection/synchronized coupling is optimal
amongst immersed couplings.
\subsection{Rate of optimal immersed coupling}
We now elicit the rate at which the reflection/synchronized coupling
occurs. This is accomplished by calculating the moment generating
function of the optimal immersed coupling time\vspace*{1pt} in the non-singular case
of $S_0\neq\widetilde{S}_0$,
supposing (without loss of generality) that $B_0>\widetilde{B}_0$.

%
\begin{thm}\label{thm:rates}
Let $T_{\mathrm{couple}}$ be the coupling time for Brownian motion
together with local time, equivalently for Brownian motion $B$
together with its supremum $S$,
using the reflection/synchronized coupling described above.
Let $\widetilde{B}$ and $\widetilde{S}$ be the corresponding
coupled quantities. Under the non-singular conditions $
S_0=B_0>\widetilde{B}_0=\widetilde{S}_0$,
the coupling time has the following moment generating function:
\[
\Expect\bigl[\exp(-\alpha T_{\mathrm{couple}})\bigr] = 1 + \sinh \biggl(\sqrt{
\frac{\alpha}{2}} (B_0-\widetilde {B}_0) \biggr)\log
\tanh \biggl(\sqrt{\frac{\alpha}{2}} \frac
{B_0-\widetilde{B}_0}{2} \biggr).
\]
\end{thm}

\begin{pf}
Recall the notation of Section~\ref{sec:algorithm}, and bear in mind
the stipulation that $S_0\neq\widetilde{S}_0$.
It is required to calculate the moment generating function
\[
\Expect\bigl[\exp(-\alpha T_{\mathrm{couple}})\bigr] = \Expect\bigl[\exp(-\alpha
T_2)\bigr] .
\]
%
We can express $T_2$ as the sum of (a) the Brownian hitting time $
T_1$, being the time taken for $B$ to pass from $B_0$ to $\tfrac
{1}{2}(B_0+\widetilde{B}_0)$,
and (b)
a randomized Brownian hitting time $T_2-T_1$, being the time taken
for $B$ to pass from $\tfrac{1}{2}(B_0+\widetilde{B}_0)$
to $M_1=\sup\{B_t\dvt t\leq T_1\}$.
Then
\[
T_2 = T_1+H^1 \bigl(\max\{S_0
\vee\widetilde{S}_0, M_1\}-\tfrac
{1}{2}(B_0+
\widetilde{B}_0) \bigr) ,
\]
where $H^1(a)$ is the time taken for a standard Brownian motion to
pass from $0$ to $a$.\vspace*{1pt}

We outline the calculations for the special case $B_0=S_0$ and $
\widetilde{B}_0=\widetilde{S}_0$ (though the calculations can be
extended to the general case).
Thus, we are concerned with the moment generating function of $
T_1+H^1(M_1-\frac{1}{2}(B_0+\widetilde{B}_0))$.

The first task is to investigate aspects of the joint distribution of $
T_1$ and $M_1$, specifically
\[
Q_\alpha(a) = \Expect\bigl[\exp(-\alpha T_1);
M_1<a+B_0\bigr] .
\]
This can be calculated using excursion theory, for example by adapting
the calculations of Rogers and Williams \cite
{RogersWilliams-1987}, \S56.
For convenience, we set $\alpha^*=\sqrt{2\alpha}$ and $
b=B_0-\frac{1}{2}(B_0+\widetilde{B}_0)=\frac{1}{2}(B_0-\widetilde
{B}_0)$.
Suppose that excursions are marked using an independent Poisson($
\alpha$) point process on the time axis. Then we distinguish
the following kinds of excursions of $B$ from $B_0$, noting the
rates at which they happen when the excursions are viewed as points of
a Poisson process of excursions with respect to
local time:
\renewcommand\theenumi{\arabic{enumi}}
\renewcommand\labelenumi{\arabic{enumi}.}
\begin{enumerate}[5.]
\item\label{item:upmarked} upward excursions which do not rise above the level $a+B_0$ but
are marked (occurring at rate $\tfrac{1}{2}(\alpha^*\coth(\alpha^*
a)-\frac{1}{a})$);
\item\label{item:farup} upward excursions which rise above the level $a+B_0$ (occurring
at rate $\tfrac{1}{2a}$);
\item\label
{item:downmarked} downward excursions which do not fall below the level $\tfrac
{1}{2}(B_0+\widetilde{B}_0)$ but are marked (occurring at rate $
\tfrac{1}{2}(\alpha^*\coth(\alpha^* b)-\tfrac{1}{b})$);
\item\label{item:fardown} downward excursions which fall below the level $\tfrac
{1}{2}(B_0+\widetilde{B}_0)$ (occurring at rate $\tfrac{1}{2b}$);
\item\label
{item:fardownmarked} downward excursions which fall below the level $\tfrac
{1}{2}(B_0+\widetilde{B}_0)$ but are not marked {before} hitting $
\tfrac{1}{2}(B_0+\widetilde{B}_0)$ (occurring at rate
$\tfrac{\alpha^*}{2}\operatorname{cosech}(\alpha^* b)$).
\end{enumerate}
These rates are computed as in the discussion of Rogers and Williams \cite{RogersWilliams-1987}, Section~56, based on
the identification of the
law of the Brownian excursion discussed there.
Thus, $Q_\alpha(a)$ can be computed as the probability that we see a
downward excursion falling to level $\frac{1}{2}(B_0+\widetilde
{B}_0)$ {before} it has been marked
(that is, an excursion of type \ref{item:fardownmarked}) \emph
{before} we ever see
upward excursions rising to level $a+B_0$ (of type \ref{item:farup}),
or staying below this level but marked (type \ref{item:upmarked}),
or downward excursions which do not fall below the level $\frac
{1}{2}(B_0+\widetilde{B}_0)$ but are marked (type \ref{item:downmarked}),
or downward excursions which fall below the level $\frac
{1}{2}(B_0+\widetilde{B}_0)$ but are marked {before} hitting $\frac
{1}{2}(B_0+\widetilde{B}_0)$
(type \ref{item:fardown} but not type \ref{item:fardownmarked}).

We can therefore use Poisson point process theory to compute
%
\begin{eqnarray}
\label{eq:mgf} \hspace*{-15pt}Q_\alpha(a) &=& \frac{\sklfrac{\alpha^*}{2}\operatorname{cosech}(\alpha^* b)}{
\afrac{1}{2a}+
\sklfrac{1}{2}(\alpha^*\coth(\alpha^* a)-\sfrac{1}{a})+
\sklfrac{1}{2}(\alpha^*\coth(\alpha^* b)-\sfrac{1}{b})+
(\afrac{1}{2b})
} \nonumber
\\[-8pt]\\[-8pt]
&=& \frac{\sinh(\alpha^* a)}{\sinh(\alpha^* (a+b))}.\nonumber
\end{eqnarray}
Consequently, we can show that the desired moment generating function
is given by
\begin{eqnarray*}
&&\Expect\bigl[\exp (-\alpha T_{\mathrm{couple}} )\bigr] \\
&&\quad = \Expect\biggl[\exp
\biggl(-\alpha\biggl(T_1+H^1\biggl(M_1-
\frac{1}{2}(B_0+\widetilde {B}_0)\biggr)\biggr)
\biggr)\biggr]
\\
&&\quad = \int_0^\infty\Expect\biggl[\exp \biggl(-\alpha
H^1\biggl(a+B_0-\frac
{1}{2}(B_0+
\widetilde{B}_0)\biggr)\biggr) \biggr]Q_\alpha(\d a)
\\
&&\quad = \int_0^\infty\Expect\bigl[\exp \bigl(-\alpha
H^1(a+b) \bigr)\bigr]Q_\alpha '(a)\,\d a
\\
&&\quad = \alpha^* \sinh\bigl(\alpha^* b\bigr) \int_0^\infty
\frac{\Expect[\exp
(-\alpha H^1(a+b) )]
}{\sinh^2(\alpha^* (a+b))} \,\d a
\\
&&\quad = \alpha^* \sinh\bigl(\alpha^* b\bigr) \int_0^\infty
\frac{\mathrm{e}^{-\alpha^*(a+b)}
}{\sinh^2(\alpha^* (a+b))} \,\d a = 1 + \sinh\bigl(\alpha^*b\bigr)\log\tanh \biggl(
\frac{1}{2}\alpha^*b \biggr)
\end{eqnarray*}
%
(using $b=\frac{1}{2}(B_0-\widetilde{B}_0)$).
\end{pf}

These excursion-theoretic calculations have been checked by simulation,
although direct simulation is computationally demanding
because of the heavy tails of the Brownian hitting times involved in
the reflection/synchronized coupling. This can to some extent
be mitigated by the use of ``Rao--Blackwellization'' using the known
formula for the moment generating function of the Brownian
first passage time.

\subsection{Comparison with maximal coupling}
It is natural to ask whether the optimal immersed coupling is in fact a
maximal coupling.
Numerical evidence is that this is not the case, as is readily seen by
computing the total mass of the minimum of the joint densities
for $(B_t,S_t)$ and $(\widetilde{B}_t,\widetilde{S}_t)$, and
comparing with the coupling probability for optimal immersion coupling.
From Revuz and Yor \cite{RevuzYor-1991}, Example~3.14
part $2^o$,
the joint density of $B_t$ and $S_t$ (supposing $S_0=B_0=0$) is
given by
%
\begin{equation}
\label{eq:mgf-maximal} \sqrt{\frac{2}{\uppi t}}\frac{2s-b}{t}\exp \biggl(-
\frac
{(2s-b)^2}{2t} \biggr)\qquad  \mbox{for } b\leq s, s>0 .
\end{equation}
(This follows readily from the reflection principle.) We can use this
to evaluate numerically the moment generating function of the maximal
coupling times
(based on the maximal coupling derived from the
method suggested by
Sverchkov and Smirnov \cite{SverchkovSmirnov-1990}).
%
A numerical comparison of the moment generating functions for these couplings
shows
that they do not have the same distribution. This is numerical evidence
that the optimal immersed coupling is distinct from (and slower than)
the maximal coupling.

An explicit non-immersion maximal coupling for Brownian motion
together with local time, for particular sets of randomized initial
conditions, can be constructed by adapting the recipe of Pitman \cite{Pitman-1976}
for discrete time and space,
and using the $2M-X$ theorem (Pitman \cite
{Pitman-1975}) to reduce the
construction to the case of the
scalar diffusion which is the
three-dimensional Bessel process.



\section{Equi-filtration couplings of Brownian motion together with~local~time}\label{sec:equi-filtration}

As noted in Section~\ref{sec:algorithm}, the reflection/synchronized
coupling is an {equi-filtration} coupling when viewed as a coupling
for the absolute value $|X|$ of Brownian motion and $L^{(0)}$ its
local time at~$0$. However
the coupling is not {equi-filtration} when
it is required to couple not just the absolute value $|X|$
but also $X$, the Brownian motion itself. This follows immediately
from consideration of the stochastic differential equation
\eqref{eq:weak-not-strong}, which is of Tanaka type when viewed as
generating the coupled Brownian motion $\widetilde{X}$ from $\int
\sgn(X)\,\d X$.
Unless $X$ and $\widetilde{X}$ are identical, it is not possible
to extract statistically appropriate
signs for the excursions of $\widetilde{X}$ from the natural filtration
$\{\mathcal{F}_t\dvt t\geq0\}$ of $X$.

The coupling of $(X,L^{(0)})$ with $(\widetilde{X},\widetilde
{L}^{(0)})$ can be modified to be equi-filtration
by replacing $\sgn(\widetilde{X}_t)$ and $\sgn({X}_t)$ in the
coupling control by time-delayed versions,
at the price of delaying the time of successful coupling. Given
a positive continuous non-increasing function $\psi\dvtx [0,\infty)\to
[0,\infty)$,
we introduce the delayed time-change
%
\begin{equation}
\label{eq:delay} \sigma(t) = t -\bigl(\psi(t)\wedge t\bigr) .
\end{equation}
%
%
For $t>0$ we define a new coupled Brownian motion $\hat{X}_t$,
starting at $\hat{X}_0=\widetilde{X}_0$ but defined
up to time $T_2$ as the solution to a time-delayed
version of \eqref{eq:weak-not-strong}:
%
\begin{equation}
\label{eq:time-delayed-sde} \d\hat{X}_t = \sgn(\hat{X}_{\sigma(t)})
J_t \sgn({X}_{{\sigma
}(t)}) \,\d X_t .
\end{equation}
Here we use \eqref{eq:co-adapted-coupling} to define the control $J$
in terms of $X$ \emph{via} the stopping time $T_1$:
\[
J_t = \cases{ -1 & \quad \mbox{if} $t \leq T_1$,
\cr
+1 &\quad
\mbox{otherwise}. }
\]
Of course this exploits the remark in Section~\ref{sec:algorithm},
that $T_1$ and $T_2$ can be defined as hitting times
for $B=L^{(0)}-|X|$.
Adopting the usual convention that $\sgn(0)=1$,
and arguing from the positivity and continuity of $\psi$,
we can solve the stochastic differential equation \eqref{eq:time-delayed-sde}
step by step over successive small time intervals. This ensures that
$\sgn(\hat{X}_{\sigma(t)})$ and $\sgn({X}_{\sigma(t)})$
are defined and
measurable with respect to $\mathcal{F}_{\sigma(t)}=\sigma\{
X_s\dvt s\leq\sigma(t)\}$.

This construction makes it clear that the coupled $\hat{X}$ is
immersed in the natural filtration of $X$. However the reverse
also holds:

%
\begin{lem}\label{lem:isofiltration}
For $\hat{X}$ defined in terms of $X$ using the time-delayed
stochastic differential equation \eqref{eq:time-delayed-sde},
it is the case that $X$ is immersed in the natural filtration of $
\hat{X}$, so that the coupling of $X$ and $\hat{X}$ is
{equi-filtration}.
\end{lem}

\begin{pf}
The control $J$ appearing in \eqref{eq:time-delayed-sde} satisfies $
J\equiv-1$ up to the stopping time $T_1$.
Accordingly $\hat{X}\equiv\hat{X}^*$ up to $T_1$, where
\[
\d\hat{X}^*_t = -\sgn\bigl(\hat{X}^*_{\sigma(t)}\bigr)
\sgn({X}_{{\sigma
}(t)}) \,\d X_t .
\]
However, we may rewrite this last stochastic differential equation as
\[
\d X_t = -\sgn({X}_{{\sigma}(t)}) \sgn\bigl(\hat{X}^*_{\sigma(t)}
\bigr) \,\d \hat{X}^*_t ,
\]
and so we find that
\[
\d X_t = -\sgn({X}_{{\sigma}(t)}) \sgn(\hat{X}_{\sigma(t)}) \,\d
\hat{X}_t
\]
holds up to time $T_1$. Since $T_1$ is a hitting time of $X$, it
follows that $X$ stopped at time $T_1$ is
adapted to the filtration of $\hat{X}$, and thus that $T_1$ and
thus $J$ are adapted to the natural filtration of $\hat{X}$.

Arguing from time $T_1$ onwards, since $J_t=1$ for $t>T_1$, we
can re-write \eqref{eq:time-delayed-sde} as
\[
\d X_t = \sgn({X}_{{\sigma}(t)}) J_t \sgn(
\hat{X}_{\sigma(t)})\, \d \hat{X}_t = \sgn({X}_{{\sigma}(t)}) \sgn(
\hat{X}_{\sigma(t)})\, \d\hat{X}_t\qquad  \mbox{for } T_1<t\leq
T_2 .
\]
It follows that $X$ up to time $T_2$ is adapted to the natural
filtration of $\hat{X}$. This establishes the mutual immersion property.
\end{pf}

Of course $\hat{X}\neq\widetilde{X}$, and therefore
we cannot assume that the {equi-filtration} coupling will have
succeeded by time $T_2$.
However,
we can use the properties of the reflection/synchronized coupling
to argue not only that
the paths of $|\widetilde{X}|$ and $|\hat{X}|$ will be close in
probability,
but also that the same is true of
the respective local times $\widetilde{L}^{(0)}$ and $\hat
{L}^{(0)}$.

%
\begin{lem}\label{lem:isofiltration-approx}
There is a sequence of {equi-filtration} couplings of $(\hat
{X}^{(n)},\hat{L}^{(0);(n)})$ with $(X, L^{(0)})$
such that $(|\hat{X}^{(n)}|,\hat{L}^{(0);(n)})$ converges in
probability under supremum norm to $(|\widetilde{X}|,\widetilde
{L}^{(0)})$
(using the reflection/synchronized coupling):
\[
\Prob\Bigl[ \sup_t \bigl\{ \bigl\llvert \bigl|
\hat{X}^{(n)}_t\bigr|-|\widetilde{X}_t|\bigr\rrvert +
\bigl\llvert \hat {L}^{(0);(n)}_t-\widetilde{L}^{(0)}_t
\bigr\rrvert \bigr\} >4^{-n} \Bigr] < 4^{-n} .
\]
\end{lem}

It suffices to exhibit a sequence for which the supremum converges to $
0$ in probability.
Note that we do \emph{not} assert that $\hat{X}^{(n)}$ converges to
$\widetilde{X}$ in probability:
if we were able to remove absolute values then this
would contradict the known fact that the stochastic differential
equation \eqref{eq:weak-not-strong} of Tanaka type
can have no strong solutions.

\begin{pf*}{Proof of Lemma~\ref{lem:isofiltration-approx}}
To simplify notation, we take $X_0=0$.
Note also that by construction $\hat{L}^{(0)}_0=\widetilde
{L}^{(0)}_0$
and $\hat{X}_0=\widetilde{X}_0$.

We shall select $\psi$ depending on $\varepsilon$ such that a
suitably fast convergent sequence $\varepsilon_n\to0$
delivers the required $(\hat{X}^{(n)},\hat{L}^{(0);(n)})$.

First note that, if we set
$\hat{S}=\hat{L}^{(0)}$ and $\hat{B}=\hat{L}^{(0)}-|\hat{X}|$,
then
\[
\d\hat{B} = \sgn(\hat{X})\sgn(\hat{X}_\sigma) J \sgn(X_\sigma )
\,\d X .
\]
By choice of initial conditions, $\hat{B}_0=\widetilde{B}_0$.
It suffices to show convergence in probability of $\sup_t\{|\hat
{B}_t-\widetilde{B}_t|\}$, the supremum norm of the difference.
Because $|J|=|\sgn(X)|=|\sgn(\hat{X})|=1$ we can use Doob's $
L^2$ submartingale inequality,
the $L^2$ isometry for Brownian stochastic integrals,
and Jensen's inequality
to deduce that
\begin{eqnarray*}
&&\Expect\Bigl[\sup_t\bigl\{(\hat{B}_t-
\widetilde{B}_t)^2\bigr\}\Bigr] \\
&&\quad = \Expect\biggl[\sup
_t \biggl\{ \biggl(\int_0^t
\bigl( \sgn(\hat{X})\sgn(\hat{X}_\sigma)J\sgn(X_\sigma) - J
\sgn(X) \bigr) \,\d X \biggr)^2 \biggr\}\biggr]
\\
&&\quad \leq 4\Expect\biggl[\int_0^\infty\bigl| \sgn(
\hat{X}_t)\sgn(\hat{X}_{\sigma(t)}) \sgn({X}_{\sigma(t)}) -
\sgn({X}_{t}) \bigr|^2\,\d t\biggr]
\\
&&\quad \leq 8{\int_0^\infty\Expect\bigl[\bigl|\sgn(
\hat{X}_{t})\sgn(\hat{X}_{\sigma
(t)}) - 1\bigr|^2\bigr]\,\d t}
+ 8{\int_0^\infty\Expect\bigl[\bigl|
\sgn({X}_{\sigma(t)}) - \sgn({X}_t) \bigr|^2\bigr]\,\d t}
\\
&&\quad = 32{\int_0^\infty\Prob\bigl[\sgn(
\hat{X}_{t})\neq\sgn(\hat{X}_{\sigma
(t)})\bigr]\,\d t} + 32{\int
_0^\infty\Prob\bigl[\sgn({X}_{\sigma(t)})\neq
\sgn({X}_t)\bigr]\,\d t} . 
%
\end{eqnarray*}
%

Both $\hat{X}$ and $X$ are Brownian motions, though not
necessarily begun at $0$.
It is therefore immediate that $\Prob[\sgn(\hat{X}_{t})\neq\sgn
(\hat{X}_{\sigma(t)})]$ and $\Prob[\sgn({X}_{\sigma(t)})\neq
\sgn({X}_t)]$
are both dominated by $\Prob[\sgn(\overline{X}_{t})\neq\sgn
(\overline{X}_{\sigma(t)})]$ for $\overline{X}$ a standard
Brownian motion begun at $0$
(since on average $\hat{X}$ will have to travel further to change
sign than the standard Brownian motion $\overline{X}$).
Moreover rotational symmetry of the standard bivariate normal
distribution reveals, if $t>\psi(t)$,\vspace*{-1pt}
\begin{eqnarray*}
\Prob\bigl[\sgn(\overline{X}_{\sigma(t)})\neq\sgn(\overline{X}_t)
\bigr] &= &\Prob\bigl[\sgn(\overline{X}_{\sigma(t)})\neq\sgn\bigl(\overline
{X}_{\sigma(t)}+(\overline{X}_t-\overline{X}_{\sigma(t)})\bigr)
\bigr]
\\[-0.5pt]
&=& \frac{1}{2}\Prob\bigl[ |\overline{X}_{\sigma(t)}| < |
\overline{X}_t-\overline{X}_{\sigma(t)}| \bigr]\\[-0.5pt]
& = &\frac{1}{2}
\Prob\biggl[ \frac{|\overline{X}_{\sigma(t)}|/\sqrt{t-\psi(t)}}{|(\overline
{X}_t-\overline{X}_{\sigma(t)})|/\sqrt{\psi(t)}} < \sqrt{ \frac{\psi(t)}{t-\psi(t) } } \biggr]
\\[-0.5pt]
&=& \frac{1}{\uppi}\tan^{-1}\sqrt{ \frac{\psi(t)}{t-\psi(t)} } .
\end{eqnarray*}
Thus we obtain, for $\varepsilon>\psi(\varepsilon)$,\vspace*{-1pt}
\begin{eqnarray*}
\Expect\Bigl[\sup_t\bigl\{(\hat{B}_t-
\widetilde{B}_t)^2\bigr\}\Bigr] &\leq& 64 \int
_0^\infty\Prob\bigl[\sgn(\overline{X}_{\sigma(t)})
\neq\sgn (\overline{X}_t)\bigr]\,\d t
\\[-0.5pt]
&\leq& 64 \biggl(\varepsilon+ \frac{1}{\uppi} \int_{\varepsilon}^\infty
\tan^{-1}\sqrt{ \frac{\psi(t)}{t-\psi(t)} } \,\d t \biggr) \\[-0.5pt]
&\leq& 64 \biggl(
\varepsilon+ \frac{1}{\uppi}\int_{\varepsilon}^\infty \sqrt{
\frac{\psi(t)}{t-\psi(t)} } \,\d t \biggr) .
\end{eqnarray*}
For any $\varepsilon\in(0,\frac{1}{2})$, we set\vspace*{-1pt}
\[
\psi(t) = \cases{ \varepsilon^3 & \quad \mbox{if} $t\in[0,\varepsilon)$,
\cr
\displaystyle \frac{\varepsilon^3}{(t-\varepsilon+1)^3} & \quad \mbox{if} $t\geq \varepsilon$. }
\]
Then (a) $\psi$ is positive continuous non-decreasing over $
[0,\infty)$, (b) $\psi(t)\leq t$ for $t\geq\varepsilon$,
and hence~(c)\vspace*{-1pt}
\begin{eqnarray*}
\Expect\Bigl[\sup_t\bigl\{(\hat{B}_t-
\widetilde{B}_t)^2\bigr\}\Bigr] &\leq& 64 \biggl(1 +
\frac{1}{\uppi}\int_{\varepsilon}^\infty \frac{1}{\sqrt{
(t/\varepsilon) (t-\varepsilon+1)^3-\varepsilon^2}
}
\,\d t \biggr)\varepsilon
\\[-0.5pt]
&\leq& 64 \biggl(1 + \frac{2}{\uppi}\int_{0}^\infty
\frac{1}{\sqrt{
4(u+1)^3-1}
} \,\d u \biggr)\varepsilon = 105.557\ldots  \times\varepsilon.
\end{eqnarray*}
The result follows, because $L^2$ convergence of random variables
implies convergence in probability.
\end{pf*}

%
\begin{thm}\label{thm:isofiltration-coupling}
There are successful {equi-filtration} couplings of Brownian motion
together with local time starting from any pair of initial conditions
$(X_0,L^{(0)}_0)$ and $(\widetilde{X}_0,\widetilde{L}^{(0)}_0)$.
\end{thm}

The proof makes it plain that the coupling time for reflection/synchronized coupling can be approximated arbitrarily well in distribution
by the coupling times for suitable {equi-filtration} couplings.

\begin{pf*}{Proof of Theorem~\ref{thm:isofiltration-coupling}}
It suffices to show that there is a sequence $\varepsilon_n\to0$
such that the following holds for all reflection/synchronized couplings:
if $\llvert |X_0|-|\widetilde{X}_0|\rrvert +\llvert L^{(0)}_0-\widetilde{L}^{(0)}_0\rrvert <\varepsilon_n$ and $
X_0=0$ then
%
\begin{equation}
\label{eq:couple-upper-bound} \Prob\bigl[T_{\mathrm{couple}}>4^{-n}\bigr]
\leq4^{-n} .
\end{equation}

For then we may construct a successful {equi-filtration} coupling as
the concatenation of a series of {equi-filtration} couplings.
In the first stage, Lemma~\ref{lem:isofiltration-approx} can be used
to select an {equi-filtration} coupling producing $(\hat{X},\hat
{L}^{(0)})$, which approximates a reflection/synchronized coupling
producing $(\widetilde{X},\widetilde{L}^{(0)})$
(continued up to but not including time $T_2^{(1)}$, the end of the
synchronized stage) such that
we have the following control on the left-limits $\hat
{X}_{T_2^{(1)}-}$, $\widetilde{X}_{T_2^{(1)}-}$, $\hat
{L}^{(0)}_{T_2^{(1)}-}$ and $\widetilde{L}^{(0)}_{T_2^{(1)}-}$:
\[
\Prob\bigl[\bigl\vert |\hat{X}_{T_2^{(1)}-}|-|\widetilde {X}_{T_2^{(1)}-}|
\bigr\vert + \bigl\llvert \hat {L}^{(0)}_{T_2^{(1)}-}-\widetilde{L}^{(0)}_{T_2^{(1)}-}
\bigr\rrvert > \varepsilon_1\bigr] < 4^{-1} .
\]
Moreover, it follows from the construction of the reflection/synchronized coupling that\linebreak[4]  $|{X}_{T_2^{(1)}-}|=|\widetilde
{X}_{T_2^{(1)}-}|=0$
while ${L}^{(0)}_{T_2^{(1)}-}=\widetilde{L}^{(0)}_{T_2^{(1)}-}$.

At time $T_2^{(1)}$, $\widetilde{X}$ makes a small jump so that $
\widetilde{X}_{T_2^{(1)}}=\hat{X}_{T_2^{(1)}-}$, while $X$, $
L^{(0)}$, $\widetilde{L}^{(0)}$, $\hat{X}$ and $\hat
{L}^{(0)}$ trajectories remain continuous. The second and further
stages are implemented by repeating the construction.

To be explicit, the construction continues through further stages $
n=2, 3, \ldots\,$, such that at the end $T_2^{(n)}-$ of stage $n$,
conditional on successful fulfilment of all previous stages, we have
$|{X}_{T_2^{(n)}-}|=|\widetilde{X}_{T_2^{(n)}-}|=0$ and $
{L}^{(0)}_{T_2^{(n)}-}=\widetilde{L}^{(0)}_{T_2^{(n)}-}$, moreover if
$n\geq2$ then $T_2^{(n)}-T_2^{(n-1)}<4^{-n}$,
and
%
\begin{equation}
\label{eq:between-stages} \Prob\bigl[\bigl\vert |\hat{X}_{T_2^{(n)}-}|-|\widetilde
{X}_{T_2^{(n)}-}|\bigr\vert + \bigl\llvert \hat {L}^{(0)}_{T_2^{(n)}-}-
\widetilde{L}^{(0)}_{T_2^{(n)}-}\bigr\rrvert > \varepsilon_n
\bigr] < 4^{-n} .
\end{equation}
(We suppress the conditioning on previous stages for the sake of simple
notation.)
Stage $n$ is implemented
(a) by using a reflection/synchronized coupling of $(|\widetilde
{X}|, \widetilde{L}^{(0)})$ with $(|X|, L^{(0)})$ which
succeeds before time $4^{-n}$ (this has probability $1-4^{-n}$),
and also (b) by
invoking Lemma~\ref{lem:isofiltration-approx} to continue $(\hat
{X},\hat{L}^{(0)})$ by
an {equi-filtration} coupling such that the maximum difference over
this stage
between the immersed coupling component $(|\widetilde{X}|, \widetilde
{L}^{(0)})$ and the {equi-filtration} coupling component $(|\hat
{X}|,\hat{L}^{(0)})$ is less than $\varepsilon_{n}$
(with conditional probability at least $1-4^{-n}$).
It follows that the conditional probability of the $n$th stage ($
n\geq2$) completing successfully is at least $1-2\times4^{-n}$.
At the end of stage $n,$ we impose a small jump on $\widetilde{X}$
so that $\widetilde{X}_{T_2^{(n)}}=\hat{X}_{T_2^{(n)}-}$.


Thus, with total probability at least
\[
1- 4^{-1} - \sum_{n=2}^\infty2
\times4^{-n} = \frac{7}{12} ,
\]
for all $n$,
stage $n$ is of duration less than $4^{-n}$,
and at time $T_2^{(n)}$
we have $X_{T_2^{(n)}-}=\widetilde{X}_{T_2^{(n)}-}=0$, $
L^{(0)}_{T_2^{(n)}-}=L^{(0)}_{T_2^{(n)}-}$, and
$(X_{T_2^{(n)}-}, L^{(0)}_{T_2^{(n)}-})$ can be approximated by the
{equi-filtration} coupling
component $(\hat{X}_{T_2^{(n)}-}, \hat{L}^{(0)}_{T_2^{(n)}-})$ with
sum of absolute differences less than $\varepsilon_{n}$.

It follows that the bound \eqref{eq:couple-upper-bound} can be used
together with Lemma~\ref{lem:isofiltration-approx} to construct an
{equi-filtration} coupling which has probability at least $\tfrac
{7}{12}$ of succeeding in finite time. In case of default at any
stage, one can
then restart the sequence of couplings, so successful immersed coupling
is almost sure to happen.

We now need to show that we can choose a sequence of $\varepsilon_n$
to satisfy the bound \eqref{eq:couple-upper-bound}.
We refer to $(B,S)$ coordinates.
From $X_0=0,$ we obtain $S_{0}=B_{0}$ while $\llvert |X_{0}|-|\hat
{X}_{0}|\rrvert =\hat{S}_{0}-\hat{B}_{0}$
and $\llvert L^{(0)}_{0}-\hat{L}^{(0)}_{0}\rrvert =\llvert S_{0}-\hat
{S}_{0}\rrvert $.
Suppose that $\llvert |X_{0}|-|\hat{X}_{0}|\rrvert <\varepsilon_n$
and $\llvert L^{(0)}_{0}-\hat{L}^{(0)}_{0}\rrvert <\varepsilon_n$.
Simple coupling arguments now show that the time taken to achieve
reflection/synchronized coupling from such starting points is
maximized if
\begin{eqnarray*}
L^{(0)}_{0}-\hat{L}^{(0)}_{0} &=&
S_{0}-\hat{S}_{0} = \varepsilon_n ,
\\
\bigl\vert |X_{0}|-|\hat{X}_{0}|\bigr\vert &=&
\hat{S}_{0}-\hat{B}_{0} = \varepsilon_n .
\end{eqnarray*}
But with such initial conditions, we can apply a scaling argument to
show that $T_2/\varepsilon^2$ has a distribution not depending on $
\varepsilon_n$.
It therefore follows that we can choose $\varepsilon_n$ to ensure
that \eqref{eq:couple-upper-bound} holds whenever
$\llvert |X_0|-|\widetilde{X}_0|\rrvert +\llvert L^{(0)}_0-\widetilde
{L}^{(0)}_0\rrvert <\varepsilon_n$ and $X_0=0$.
%
%
\end{pf*}


The following corollary will be of assistance when constructing
{equi-filtration} couplings of BKR diffusions in the next section.

\begin{cor}\label{cor:finite-range}
The {equi-filtration} coupling of Theorem~\ref
{thm:isofiltration-coupling} can be localized in the following sense:
for fixed $\delta>0$, for all sufficiently small $\varepsilon>0$
and for all pairs of initial conditions
$(X_0,L^{(0)}_0)$ and $(\widetilde{X}_0,\widetilde{L}^{(0)}_0)$ with
\[
\bigl\vert |X_0|-|\widetilde{X}_0|\bigr\vert +\bigl\llvert
L^{(0)}_0-\widetilde {L}^{(0)}_0\bigr
\rrvert <\varepsilon,
\]
we can construct a successful {equi-filtration} coupling of $
(X,L^{(0)})$ and $(\widetilde{X},\widetilde{L}^{(0)})$
such that
\[
\Prob\bigl[\mbox{one of } \bigl(X,L^{(0)}\bigr) \mbox{ and } \bigl(
\widetilde {X},\widetilde{L}^{(0)}\bigr) \mbox{ does not stay within }
\ball\bigl(\bigl(X_0,L^{(0)}_0\bigr), \delta
\bigr)\bigr] < \varepsilon.
\]
\end{cor}

\begin{pf}
From the proof of Theorem~\ref{thm:isofiltration-coupling}, it follows
that for sufficiently small $\varepsilon$ we can obtain
uniformly arbitrarily small probability of the {equi-filtration}
coupling failing to couple within an arbitrarily small period of time.
The corollary then follows by observing that it follows from continuity
of Brownian motion and Brownian local time that over a sufficiently
small period of time we can ensure that the probability of there being
large deviations either of motion or of local time is arbitrarily small.
\end{pf}

\section{Application to BKR diffusions}\label{sec:BKR}
The BKR diffusion was introduced by Bene\v{s}, Karatzas and Rishel \cite{BenesKaratzasRishel-1991}
as part of an investigation into a control problem
which possesses no strict-sense optimal law: it is a two-dimensional
diffusion $(X,Y)$ for which $X$ and $Y$ are Brownian motions
connected by
%
\begin{equation}
\label{eq:BKR-relation} \sgn(X)\,\d X + \sgn(Y)\,\d Y = 0 .
\end{equation}
Essentially $(X,Y)$ diffuses linearly in a Brownian fashion
along the boundary of a square $\{(x,y)\dvt  |x|+|y|=\ell\}$, except
that $\ell$ increases according to a local time driving term whenever
$(X,Y)$ visits one of the vertices of the square.
\'{E}mery \cite{Emery-2009} studied filtration
questions concerning this
process, and in particular showed that the natural filtration of $
(X,Y)$ is Brownian even when $(X,Y)$ is started from
the origin. In this connection, he asked a specific question
(\'{E}mery \cite{Emery-2009}, Question), which can be stated
concisely as follows: given
two initial points $(X_0,Y_0)$ and $(\widetilde{X}_0,\widetilde
{Y}_0)$,
neither of which is the origin, is there an almost surely successful
{equi-filtration} coupling of BKR diffusions $(X,Y)$ and $
(\widetilde{X},\widetilde{Y})$ started from these two points?
An affirmative answer would lead to a constructive proof of Brownianity
of the filtration of $(X,Y)$ started from the origin.
An immersed coupling is exhibited in \'{E}mery \cite
{Emery-2009}, Lemma~5;
however {equi-filtration} is required for the purposes of obtaining a
constructive proof of the filtration result.\issue{Understand the relevant filtration theory}

\subsection{Sketch construction of immersed coupling for the BKR
diffusion}\label{sec:immersed-BKR}
The work of Sections~\ref{sec:brownian-local-time}, \ref
{sec:equi-filtration} suggests a direct strategy for constructing
successful {equi-filtration} couplings of BKR diffusions;
start with an almost-surely successful immersed coupling, and then
perturb it to produce a nearly-successful coupling which can then be
iterated to generate a successful coupling
following the methods used to prove Theorem~\ref
{thm:isofiltration-coupling} and associated lemmas.
The immersed coupling described in \'{E}mery \cite
{Emery-2009}, Lemma~5, bears a
strong family resemblance to the reflection/synchronized coupling
given above (Example~\ref{coupling:reflection/synchronized}): in
preparation for construction of the {equi-filtration} coupling
we first sketch a description of \'{E}mery's immersed coupling
for BKR diffusions using the terminology of our paper.

Application of the Tanaka formula for Brownian local time to \eqref
{eq:BKR-relation} shows that
$|X|+|Y|=L^{X;(0)}+L^{Y;(0)}$ increases as the sum of the local times
accumulated by $X$ and $Y$ at $0$.
In the case when $(X,Y)$ does not begin at the origin, so that $
h=\tfrac{1}{2}(|X_0|+|Y_0|)>0$, we can
determine a real Brownian motion which drives the BKR diffusion by
first defining a binary c\'adl\'ag switch process $K$, which
takes values $0$ or $1$, changing only according to the following rules:
\begin{enumerate}[2.]
\item[1.]$K$ takes value $1$ on entry to the region $|X|<h$;
\item[2.]$K$ takes value $0$ on entry to the region $|Y|<h$;
\end{enumerate}
and otherwise $K$ is time-constant.
The initial value of $K$ is given by
\[
K_0 = \cases{ 1 & \quad \mbox{if} $|X_0|\leq h$,
\cr
0 & \quad \mbox{otherwise}. }
\]
This construction
is illustrated in Figure \ref{fig:regions}.

We can then define a real-valued Brownian motion $A$ by
%
\begin{equation}
\label{eq:driving-BM} \d A = K \sgn(Y)\,\d X - (1-K)\sgn(X)\,\d Y .
\end{equation}
Note that the definition of $K$ implies that $Y$ never vanishes
when $K=1$, and $X$ never vanishes
when $K=0$. This allows us to use $A$ to construct $(X, Y)$ as follows:
%
\begin{eqnarray}\label{sec:BKR-definition}
\d X &=& K\sgn(Y)\,\d A - (1-K)\sgn(X)\sgn(Y)\,\d Y ,
\nonumber
\\[-8pt]\\[-8pt]
\d Y &=& -K\sgn(X)\sgn(Y)\,\d X - (1-K)\sgn(X)\,\d A .\nonumber
\end{eqnarray}
Thus if $K=1,$ then we can construct $X$ in terms of $A$ (since $
Y$ does not then change sign) and then $Y$ in terms of $X$,
and conversely if $K=0,$ then we can construct $Y$ in terms of $
A$ and then $X$ in terms of $Y$.
In particular, it is convenient to note the following relationships
between $X$, $Y$ and the driving Brownian motion $A$:
%
\begin{eqnarray}
\label{eq:BKR-X-defn}\d X &=& \sgn(Y)\,\d A ,
\\
\label{eq:BKR-Y-defn}\d Y &=& -\sgn(X)\,\d A .
\end{eqnarray}

%
\begin{figure}

\includegraphics{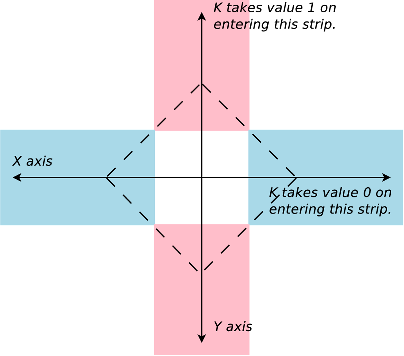}

\caption{The construction of a BKR diffusion $(X, Y)$ from a real
Brownian motion can be based on a binary c\'adl\'ag
process $K$, which switches when it enters specific regions with
boundaries determined by
$x,y=\pm h$ for $h=\tfrac{1}{2}(|X_0|+|Y_0|)$.
The initial value $(X_0,Y_0)$ lies on the boundary of the
diamond-shaped region, and
by construction the diffusion will never enter the interior of this region.}
\label{fig:regions}
\end{figure}

We can now sketch the construction of an immersed coupling with another
BKR diffusion
which also does not begin
at the origin. Let $(\widetilde{X},\widetilde{Y})$ be the coupled
BKR diffusion,
based on $\widetilde{h}=\tfrac{1}{2}(|\widetilde{X}_0|+|\widetilde
{Y}_0|)>0$,
with $\widetilde{K}$ and $\widetilde{A}$
defined in direct analogy to $K$ and $A$.
We define the coupling (a variant of the reflection/synchronized
coupling described in Definition~\ref
{coupling:reflection/synchronized}) by the following definition.

\begin{defn}[(Variant reflection/synchronized coupling)]\label
{coupling:BKR-immersed-coupling}
With the above notation, two BKR diffusions
$({X},{Y})$ and
$(\widetilde{X},\widetilde{Y})$ (adapted to the same filtration)
are said to be
in \emph{reflection/synchronized coupling} if their driving Brownian
motions $A$ and $\widetilde{A}$ are related by
%
\begin{equation}
\label{eq:BKR-immersed-coupling} \d\widetilde{A} = \sgn(\widetilde{Y}) J \sgn(Y)\,\d A ,
\end{equation}
where $J=-1$ until $X$ and $\widetilde{X}$ first meet, after
which we set $J=+1$. (Thus $\widetilde{X}$ is a reflection of $
X$ till $X=\widetilde{X}$.)
\end{defn}

Note that this variant coupling does not treat $X$ and $Y$ symmetrically.
As discussed in detail by \'{E}mery \cite{Emery-2009},
the variant reflection/synchronized coupling can be represented as an immersed coupling and is
almost surely successful.
Given $(X,Y)$, in order to reconstruct $(\widetilde{X},\widetilde
{Y})$
it is necessary to augment the filtration using an
appropriate independent sequence of equiprobable $\pm1$ random
variables; this is because it
is not possible to obtain strong solutions to the stochastic differential
system given by \eqref{eq:BKR-immersed-coupling}, \eqref
{sec:BKR-definition}, and the analogue of \eqref{sec:BKR-definition}
giving $(\widetilde{X},\widetilde{Y})$ in terms of $\widetilde
{A}$.
Note however that we \emph{can} construct the paths of $\widetilde
{X}$ and $|\widetilde{Y}|$ from the path of $(X,Y)$: indeed

%
\begin{lem}\label{lem:BKR-reconstruction}
If $(X,Y)$ and $(\widetilde{X},\widetilde{Y})$ are connected by the
variant reflection/synchronized coupling specified in Definition~\ref
{coupling:BKR-immersed-coupling} then $\widetilde{X}$ and $
|\widetilde{Y}|$
are both adapted to the natural filtration of $X$, and $T_1=\inf\{
t\dvt X_t=\widetilde{X}_t\}$,
and the coupling time $T_{\mathrm{couple}}=T_2=\inf\{
t>T_1\dvt |Y_t|=|\widetilde{Y}_t|=0\}$
are stopping times for this filtration. In particular $T_{\mathrm
{couple}}$
is almost surely finite.
\end{lem}

\begin{pf}
For the case of $\widetilde{X}$, observe that
$T_1$ is the first time $t$ that $\widetilde{X}_0-(X_{t}-X_0)$
hits $\tfrac{1}{2}(X_0+\widetilde{X}_0)$,
and therefore $T_1$ is a stopping time for the natural filtration of
$X$. Moreover,
\[
\widetilde{X} = \bigl(\widetilde{X}_0-(X_{t\wedge T_1}-X_0)
\bigr)+(X_{t\vee
T_1}-X_{T_1}) ;
\]
hence $\widetilde{X}$ is also adapted to the natural filtration of $
X$.\vspace*{1pt}

For the case of $|\widetilde{Y}|$, observe that
$|\widetilde{Y}|$ satisfies
\[
\d|\widetilde{Y}| = \sgn(\widetilde{Y})\,\d\widetilde{Y} + \d L^{\widetilde{Y};(0)} ,
\]
where $L^{\widetilde{Y};(0)}$ is the local time accumulated by $
\widetilde{Y}$ at $0$,
and we can use the L\'evy transform
to show that it suffices to establish that $\int_0^t\sgn(\widetilde
{Y})\,\d\widetilde{Y}$ is measurable with respect to the natural
filtration of $X$. But
we can employ the stochastic differential equation \eqref
{eq:BKR-immersed-coupling} determining the coupling, together with
\eqref{eq:BKR-Y-defn};
\begin{eqnarray*}
\sgn(\widetilde{Y})\,\d\widetilde{Y}& =& -\sgn(\widetilde{X})\sgn(\widetilde{Y})\,\d
\widetilde{A} = -\sgn(\widetilde{X})\sgn(\widetilde{Y}) \sgn(\widetilde{Y}) J
\sgn(Y)\,\d A
\\
&=& -\sgn(\widetilde{X}) J \sgn(Y)\,\d A .
\end{eqnarray*}
Thus, we can deduce that $\int_0^t\sgn(\widetilde{Y})\,\d\widetilde
{Y}$ is given by a Brownian stochastic integral adapted to the natural
filtration of $X$,
since $\d X=\sgn(Y)\,\d A$ by \eqref{eq:BKR-X-defn}, and we have
already shown that $\widetilde{X}$ is so adapted.

Moreover after $T_1$ we have $J\equiv1$ and $X\equiv\widetilde
{X}$, hence
\begin{eqnarray*}
\d|\widetilde{Y}| - \d|Y|& =& \sgn(\widetilde{Y})\,\d\widetilde{Y} - \sgn(Y)\,\d Y +
\d L^{\widetilde{Y}:(0)} -\d L^{{Y}:(0)}
\\
&=& -\sgn(\widetilde{X})\sgn(Y)\,\d A - \sgn(Y)\,\d Y + \d L^{\widetilde
{Y}:(0)} -\d
L^{{Y}:(0)}
\\
&=& -\sgn({X})\sgn(Y)\,\d A - \sgn(Y)\,\d Y + \d L^{\widetilde{Y}:(0)} -\d L^{{Y}:(0)}
= \d L^{\widetilde{Y}:(0)} -\d L^{{Y}:(0)} ,
\end{eqnarray*}
where we use \eqref{eq:BKR-Y-defn} to cancel the two Brownian
stochastic differentials.
Thus after $T_1$ it is the case that either $|Y|\geq|\widetilde
{Y}|$ for all time, or $|Y|\leq|\widetilde{Y}|$ for all time,
and coupling occurs at the first time $T_2$ after $T_1$ that
the coupled reflected Brownian motions $|Y|$ and $|\widetilde{Y}|$
simultaneously hit $0$.

Now $|Y|$ hits zero when $|X|$ first hits $|X_0|+|Y_0|$, or at
subsequent times when $|X|$ attains its running supremum,
while we have already shown that $|\widetilde{Y}|$ is adapted to the
natural filtration of $X$, therefore $T_2$
is a stopping time for this filtration.

Finally, almost sure finiteness of $T_{\mathrm{couple}}$ follows,
since it is the (dependent) sum of two almost surely finite Brownian
hitting times $T_1$ and $T_2-T_1$.
\end{pf}

We shall use this partial reconstruction when analyzing the
{equi-filtration} coupling described below.


\subsection{An {equi-filtration} coupling for BKR diffusions}\label
{sec:equi-filtration-BKR}
In order to construct an {equi-filtration} coupling for BKR diffusions
$(X, Y)$ and $(\widetilde{X}, \widetilde{Y})$,
with neither BKR diffusion starting at the origin, we adopt the
strategy of Section~\ref{sec:equi-filtration}.\issue{Check details
for {equi-filtration} coupling for BKR diffusion}
Given a delayed time-change $\sigma(t)$ defined in terms of a
positive continuous non-increasing function $\psi$ as in \eqref{eq:delay},
we can define a new driving Brownian motion $\hat{A}$ in terms of $
A$ via
%
\begin{equation}
\label{eq:isofiltration-BKR} \d\hat{A}_t = \sgn(\hat{Y}_{\sigma(t)})
J_t \sgn(Y_{\sigma(t)}) \,\d A .
\end{equation}
Here $X$ (and $Y$) are defined in terms of $A$ using \eqref
{sec:BKR-definition};
we note that $J$ is the \emph{immersed} control given in the
previous subsection, constructed in terms of $(X,Y)$
by setting $J=-1$ till the time $T_1$ when $X$ first hits $
\tfrac{1}{2}(X_0+\widetilde{X}_0)$,
and then setting $J=+1$; finally, $\hat{Y}$ (and $\hat{X}$)
are defined in terms of $\hat{A}$ using the analogue of \eqref
{sec:BKR-definition}. The use of the delay $\sigma$ means that
the system of these
stochastic differential equations
has a unique strong solution so long as neither BKR diffusion is begun
at the origin.
Note that there are issues in finding strong solutions to \eqref
{sec:BKR-definition} together with the switching processes
$K$ and $\hat{K}$
if either or both of the BKR diffusions start at the origin.

%
\begin{lem}\label{lem:isofiltration-BKR}
Suppose $(\hat{X}, \hat{Y})$ is a BKR diffusion defined in terms of
a BKR diffusion $(X,Y)$ using the time-delayed stochastic
differential equation \eqref{eq:isofiltration-BKR} and
the analogues of the defining equation \eqref{sec:BKR-definition}
together with switching processes $K$ and $\hat{K}$, so that
%
\begin{eqnarray}
\label{eq:BKR-hat-X-defn}\d\hat{X} &=& \sgn(\hat{Y})\,\d\hat{A} ,
\\
\label{eq:BKR-hat-Y-defn}\d\hat{Y} &=& -\sgn(\hat{X})\,\d\hat{A} .
\end{eqnarray}
If neither BKR diffusion is begun at the origin, then the resulting
coupling is
{equi-filtration}.
\end{lem}

Note that this definition is not symmetrical in $(X,Y)$ and $(\hat
{X}, \hat{Y})$, since the coupling control $J$ is defined in terms
of $(X,Y)$.
Note further that we do not assert that the coupling is successful!

\begin{pf*}{Proof of Lemma~\ref{lem:isofiltration-BKR}}
It follows from construction that $(\hat{X}, \hat{Y})$ is immersed
in the filtration of $(X, Y)$.
On the other hand we can argue as in Lemma~\ref{lem:isofiltration}
that the reverse also holds, and hence that this coupling is {equi-filtration}.
As in Lemma~\ref{lem:isofiltration} of the argument for the case of
Brownian motion with local time,
the key point is to argue first that the trajectory of $(X,Y)$ up to
the time $T_1$ (while $J=-1$) is immersed in the filtration of $
(\hat{X}, \hat{Y})$,
and then to argue that the subsequent construction (while $J=+1$) is
also immersed. The crucial point is that $T_1$ is a stopping time for
the filtration of
$(\hat{X}, \hat{Y})$, as noted in Lemma~\ref{lem:BKR-reconstruction}.
\end{pf*}

Because of Lemma~\ref{lem:BKR-reconstruction},
it makes sense to
discuss $\widetilde{X}_{T_2}$ and $|\widetilde{Y}_{T_2}|$ defined
in terms of $X$ and $Y$, and in particular to
consider the extent to which $\hat{X}_{T_2}$ and $|\hat
{Y}_{T_2}|$ differ from $\widetilde{X}_{T_2}$ and $|\widetilde
{Y}_{T_2}|$.
Moreover $|\widetilde{Y}|_{T_2}=Y_{T_2}=0$, so control of $\llvert |\hat{Y}_{T_2}|-|\widetilde{Y}_{T_2}|\rrvert $ corresponds directly
to control of $|\hat{Y}_{T_2}-\widetilde{Y}_{T_2}|=|\hat
{Y}_{T_2}|$ itself.

%
\begin{lem}\label{lem:BKR-approximation}
Suppose $(\hat{X}, \hat{Y})$ is a BKR diffusion defined in terms of
a BKR diffusion $(X,Y)$ using the time-delayed stochastic
differential equation \eqref{eq:isofiltration-BKR} and
the analogues of the defining equations \eqref{eq:BKR-X-defn} and
\eqref{eq:BKR-Y-defn}. For any $\delta>0$, we can choose $
\varepsilon\in(0,\tfrac12)$ sufficiently small so that if the
time-delay $\sigma(t)= t -(\psi(t)\wedge t)$ is defined \emph{via}
$\psi(t)={\varepsilon^3}/({(t-\varepsilon+1)^3})$ for $t\geq
\varepsilon$ then\vspace*{-1pt}
%
\begin{equation}
\label{eq:BKR-approximation} \Prob\bigl[|\hat{X}_{T_2}-\widetilde{X}_{T_2}|+|
\hat{Y}_{T_2}-\widetilde {Y}_{T_2}|>\delta\bigr] \leq\delta.
\end{equation}
\end{lem}

\begin{pf}
Consider the stochastic differential equation for $\hat{X}$:\vspace*{-1pt}
%
\[
\d\hat{X} = \sgn(\hat{Y})\,\d\hat{A} = J \sgn(Y_\sigma)\sgn(\hat
{Y}_\sigma)\sgn(\hat{Y})\,\d A .
\]
Since $\d X=\sgn(Y)\,\d A$, and since $Y\approx Y_\sigma$ and $
\hat{Y}\approx\hat{Y}_\sigma$, it follows that
the coupling between $\hat{X}$ and $X$ approximates a reflection
coupling up to time $T_1$, and after that approximates a synchronized
coupling. Calculating as
in Lemma~\ref{lem:isofiltration-approx}, but recalling from Lemma~\ref
{lem:BKR-reconstruction} that
$\widetilde{X}=(\widetilde{X}_0-(X_{t\wedge T_1}-X_0))+(X_{t\vee
T_1}-X_{T_1})$ is actually adapted to the filtration
of $(X,Y)$,\vspace*{-1pt}
\begin{eqnarray*}
\Expect\Bigl[\sup_{t} \bigl\{ (\hat{X}_t-
\widetilde{X}_t )^2 \bigr\}\Bigr] &\leq&4{\int
_0^\infty\Expect\bigl[ \bigl(J\sgn(Y_\sigma)
\sgn(\hat {Y}_\sigma)\sgn(\hat{Y})-J\sgn(Y) \bigr)^2\bigr]\,\d
t}
\\
&\leq& 32{\int_0^\infty\Prob\bigl[\sgn(
\hat{Y}_\sigma)\neq\sgn(\hat{Y})\bigr]\,\d t} + 32{\int_0^\infty
\Prob\bigl[\sgn({Y}_\sigma)\neq\sgn({Y})\bigr]\,\d t}
\\
&\leq& 105.557\ldots  \times\varepsilon 
.
\end{eqnarray*}
%
Thus we can control the extent to which the approximate reflection
coupling of $X$ and\vspace*{1pt} $\hat{X}$
will deviate from
the reflection coupling of $X$ and $\widetilde{X}$ by using the
Markov inequality: for any $\delta>0$
such that $\varepsilon<\delta^3/(12\times105.557\ldots)$, it
follows that\vspace*{1pt}
%
\begin{equation}
\label{eq:control-of-X} \Prob\Bigl[\sup_{t} \bigl\{\llvert
\hat{X}_t-\widetilde{X}_t\rrvert \bigr\}>\delta/2\Bigr]
\leq\delta/3 .
\end{equation}


We now need to control the approximation of $|\widetilde{Y}|$ by $
|\hat{Y}|$.
We first use the almost-sure finiteness of the stopping time $T_2$ to
select a constant time $t_{\max}$ such that
%
\begin{equation}
\label{eq:control-of-T} \Prob[T_2>t_{\max}]<\delta/3 .
\end{equation}
It suffices to control the approximation of $\int_0^t\sgn(\widetilde
{Y})\,\d\widetilde{Y}$ by $\int_0^t\sgn(\hat{Y})\,\d\hat{Y}$ when
$0\leq t\leq t_{\max}$.
Observe that
\begin{eqnarray*}
\sgn(\hat{Y})\,\d\hat{Y} - \sgn(\widetilde{Y})\,\d\widetilde{Y} &=& - \sgn(\hat{Y})
\sgn(\hat{X})\,\d\hat{A} + \sgn(\widetilde{Y})\sgn (\widetilde{X})\,\d\widetilde{A}
\\
&=& - J \sgn(\hat{Y})\sgn(\hat{X})\sgn(\hat{Y}_\sigma)\sgn
({Y}_\sigma)\,\d A + J \sgn({Y})\sgn(\widetilde{X})\,\d{A} .
\end{eqnarray*}

Hence, (using the same definition of $\psi$)
\begin{eqnarray*}
&&\Expect\biggl[\sup_{t\leq t_{\max}} \biggl\{ \biggl(\int
_0^t \sgn(\hat {Y})\,\d\hat{Y} - \int
_0^t \sgn(\widetilde{Y})\,\d\widetilde{Y}
\biggr)^2 \biggr\}\biggr]
\\
&&\quad \leq4 {\int_0^{t_{\max}}\Expect\bigl[ \bigl( \sgn(
\hat{Y})\sgn(\hat{X})\sgn(\hat{Y}_\sigma)\sgn({Y}_\sigma) -
\sgn({Y})\sgn(\widetilde{X}) \bigr)^2 \bigr]\,\d t}
\\
&&\quad \leq 48 {\int_0^\infty\Prob\bigl[ \sgn(\hat{Y})
\neq\sgn(\hat{Y}_\sigma) \bigr]\,\d t} + 48 {\int_0^\infty
\Prob\bigl[ \sgn({Y})\neq\sgn({Y}_\sigma) \bigr]\,\d t}
\\
&&\qquad {}+ 48 {\int_0^{t_{\max}}\Prob\bigl[ \sgn(\hat{X}) \neq
\sgn(\widetilde{X}) \bigr]\,\d t}
\\
&&\quad \leq 158.336\ldots  \times\varepsilon + 48{\int_0^{t_{\max}}
\Prob\bigl[ \sgn(\hat{X}) \neq\sgn(\widetilde{X}) \bigr]\,\d t} .
\end{eqnarray*}

Now
%
\begin{eqnarray*}
\int_0^{t_{\max}}\Prob\bigl[ \sgn(\hat{X}) \neq\sgn(
\widetilde{X}) \bigr]\,\d t &\leq& \int_0^{t_{\max}} \bigl(
\Prob\bigl[|\hat{X}-\widetilde{X}|\geq\eta\bigr] + \Prob\bigl[|\widetilde{X}|<\eta\bigr] \bigr)\, \d t
\\
&\leq& \int_0^{t_{\max}} \biggl( \frac{1}{\eta^2}
\Expect\bigl[\llvert \hat{X}-\widetilde{X}\rrvert ^2\bigr] + 1\wedge
\frac{2\eta}{\sqrt{2\uppi t}} \biggr) \,\d t
\\
&\leq& \frac{105.557\ldots\times t_{\max}}{\eta^2} \varepsilon + \frac{2}{\uppi}\eta^2 +
\sqrt{\frac{8t_{\max}}{\uppi}}\eta .
\end{eqnarray*}

It follows that if we choose first $\eta$ then $\varepsilon$
small enough that
\begin{eqnarray*}
\eta&< &\frac{1}{12\times48}\sqrt{\frac{\uppi}{8 t_{\max}}} \frac
{\delta^3}{12} ,
\\
\eta&<& \sqrt{\frac{\uppi}{96\times12} \frac{\delta^3}{12}} ,
\\
\varepsilon&<& \frac{\eta^2}{12\times48\times105.557\ldots\times
t_{\max}} \frac{\delta^3}{12} ,
\\
\varepsilon&<& \frac{1}{12\times158.336\ldots} \frac{\delta
^3}{12} ,
\end{eqnarray*}
then
%
\begin{eqnarray}
\label{eq:control-of-Y} &&\Expect\biggl[\sup_{t\leq t_{\max}} \biggl\{ \biggl(\int
_0^t \sgn(\hat {Y})\,\d\hat{Y} - \int
_0^t \sgn(\widetilde{Y})\,\d\widetilde{Y}\biggr)
^2 \biggr\}\biggr] \nonumber
\\[-8pt]\\[-8pt]
&&\quad \leq 158.336\ldots  \times\varepsilon + 48 \biggl( \frac{105.557\ldots\times t_{\max}}{\eta^2} \varepsilon +
\frac{2}{\uppi}\eta^2 + \sqrt{\frac{8t_{\max}}{\uppi}}\eta \biggr) \leq
\frac{1}{3} \times\frac{\delta^3}{12} .\nonumber
\end{eqnarray}
The lemma now follows from the inequalities \eqref{eq:control-of-X},
\eqref{eq:control-of-T}, and an application of the Markov inequality
to \eqref{eq:control-of-Y}.
\end{pf}

We can now state and prove the main result of this section, that BKR
diffusions begun at non-zero points
can be coupled in an {equi-filtration} manner.

\begin{thm}\label{thm:BKR-isofiltration-coupling}
Given two BKR diffusions begun at different non-zero initial points,
they can be coupled in a mutually immersed manner
which succeeds
in almost surely finite time,
using
an infinite sequence of increasingly rapid {equi-filtration} couplings,
each of which approximates the
variant reflection/synchronized coupling but with delays built into
the reflection
and synchronization couplings to render them {equi-filtration}.
\end{thm}

\begin{pf}
From Lemma~\ref{lem:BKR-approximation} it follows that
at time $T_2$ we can bring $(X,Y)$ and $(\hat{X},\hat{Y})$
arbitrarily close together
with probability arbitrarily close to $1$. Moreover $
Y_{T_2}=\widetilde{Y}_{T_2}=0$, so that $\hat{Y}_{T_2}$ will be
arbitrarily close to zero with probability arbitrarily close to $1$.

Restarting at time $T_2$, the evolutions of $(X,|Y|)$ and $(\hat
{X},|\hat{Y})|$ can be related to the behaviour
of Brownian motion and its local time at $0$.
Specifically, $(|Y|, X-|Y|)$ (respectively $(|\hat{Y}|, \hat
{X}-|\widetilde{Y}|)$) has the stochastic
dynamics of the absolute value of Brownian motion together with its
local time at $0$, at least until $X$ (respectively $\hat{X}$)
hits zero.

We can therefore apply the iterative coupling technique of Section~\ref
{sec:equi-filtration} to achieve exact coupling of $(X,Y)$ and $
(\hat{X},\hat{Y})$; the localization
supplied by Corollary~\ref{cor:finite-range} implies that there is a
positive probability of achieving this coupling
before either $X$ or $\hat{X}$ hit $0$, with a uniform positive
lower bound on the probability which tends to $1$ as
the
restarted values at $T_2$ of $X-\hat{X}$, $Y$ and $\hat{Y}$
tend to zero. In the event of default (i.e., the initial variant
reflection/synchronized coupling fails to achieve approximate
coupling at $T_2$, or $X$ or $\hat{X}$ hits $0$
subsequent to $T_2$), then
the whole coupling procedure
can be restarted; the initial delayed variant reflection/synchronized
coupling can be arranged to deliver approximate coupling to an
arbitrarily small tolerance with probability arbitrarily close to $1$,
and then the subsequent iterative coupling
will also have success probability arbitrarily close to $1$. Thus,
it is possible to arrange that the coupling procedure will only need to
be restarted a finite number of times before coupling is achieved.
\end{pf}


\section{Conclusion}\label{sec:conclusion}
In this paper, we have established the basic properties of the
reflection/synchronized coupling for Brownian motion
together with local time, and in particular we have shown it is an
optimal immersed. but not maximal, coupling.
We have also shown that this coupling can be approximated by explicit
{equi-filtration} couplings; moreover
that it can be used as the basis for immersion and {equi-filtration}
couplings of a more complicated diffusion,
thus answering a question arising from filtration theory.

Further questions include the following:
\begin{enumerate}[3.]
\item[1.] Is \'{E}mery's \cite{Emery-2009} immersed BKR
coupling optimal amongst all
immersed couplings of BKR diffusions? We would be surprised if this
were the case,
since there are two distinct variants of \'{E}mery's \cite
{Emery-2009} coupling
depending on whether the construction is based on the $X$ component
or the $Y$ component:
it seems implausible that optimal couplings would permit such a
symmetry. It is of course possible that there is \emph{no} optimal
immersed coupling: it may be the case
that some strategies work better for short time while others work
better for long time.
\item[2.] Reverting to the case of Brownian motion with local time, is it
possible to couple Brownian motion together with local times
accumulated at two or more distinct points?
This seems to be a hard question. Analogous generalizations have been
carried out for the case of coupling Brownian motions together with
iterated time-integrals
(Kendall and Price \cite{KendallPrice-2004}); however
there is a useful linearity in the
time-integral case which is not present here. The question of coupling
finite sets of local times could be
viewed as a question of whether one could couple a finite version of
the celebrated Brownian burglar (Warren and Yor \cite{WarrenYor-1998}, Aldous \cite{Aldous-1998}).
\item[3.] There is interesting territory to be explored in the realm of
couplings which fall short of being maximal but yet are not immersed.
One example in applied
stochastic process theory is supplied by Smith
\cite{Smith-2011a}, who
investigates the mixing time of a simple Gibbs sampler on the unit
simplex using a two-stage coupling of which the first is
immersed (Markovian, in Smith's chosen terminology)
while the second couples an associated partition process anticipatively.
This non-immersed coupling allows Smith to give an
affirmative answer to a conjecture by Aldous concerning the mixing time
of this Gibbs sampler.
Arguably Sigman's \cite{Sigman-2011} perfect
simulation algorithm for
super-stable $M/G/c$ queues can be put in the same category, as this
depends on coupling service times not according
to time of arrival of customer but according to time of start of service.
It would be likely to be most illuminating if one could discover simple
Brownian coupling problems for which gains of a similar kind can be made.
\item[4.] It seems clear that the techniques of this paper can be
generalized to show that immersion couplings of suitably regular diffusions
can always be approximated by {equi-filtration} couplings, and it would
be interesting to see a fully rigorous proof
in a case where the qualification ``suitably regular'' is given a
pleasant and natural meaning.
\end{enumerate}

The underlying reflection/synchronization coupling for Brownian
motion together
with local time is extremely simple, and lends itself to rather
complete calculation. Not only is it an example of the general programme
of coupling Brownian motion together with appropriate functionals
(Ben Arous, Cranston, and Kendall \cite{BenArousCranstonKendall-1995},
Kendall and Price \cite{KendallPrice-2004},
Kendall \cite{Kendall-2007,Kendall-2009d}),
but also it can be viewed as a basic coupling strategy that, like the
reflection coupling of Brownian motion (Lindvall
\cite{Lindvall-1982a})
has the potential to serve as a model in much more general situations.
The application to the BKR diffusion in this paper illustrates this point;
it is hoped that the calculations described above will facilitate the
use of the reflection/synchronization coupling as a building block in
other applications of coupling to probability theory.

\section*{Acknowledgements}
This work was supported in part by EPSRC Research Grant EP/K013939.
I gratefully acknowledge the contribution of Prof. Michel \'{E}mery, whose question concerning BKR diffusions encouraged me to
develop the initial
idea of coupling Brownian motion together with local time, and who
pointed out how to improve the last part of the proof of Theorem~\ref
{thm:isofiltration-coupling}.
I am also thankful for useful remarks from an anonymous referee.





\printhistory

\end{document}